\documentclass[11pt,a4paper]{article}
\usepackage[T1]{fontenc}
\usepackage[utf8]{inputenc}
\usepackage{tikz, scalerel}
\usepackage{amsmath}
\usepackage{enumerate}
\usepackage{amssymb}
\usepackage{color}
\usepackage{amsthm}
\usepackage{verbatim}
\usepackage{mathabx}
\usepackage{comment}
\usepackage{hyperref}
\usepackage{mathrsfs}
\usepackage{calligra}
\usepackage[refpage, intoc]{nomencl}
\usepackage[toc, style=super]{glossaries-extra}
\usepackage{stmaryrd}
\usepackage{a4wide}

\newtheorem{definition}{Definition}[section]
\newtheorem{theorem}[definition]{Theorem}

\newtheorem{remark}[definition]{Remark}

\newtheorem{lemma}[definition]{Lemma}

\newtheorem*{theorem*}{Theorem}

\newcommand{\ii}{\ensuremath{\imath}}
\newcommand{\RR}{\ensuremath{\mathbb{R}}}
\newcommand{\NN}{\ensuremath{\mathbb{N}}}
\newcommand{\Sw}{\ensuremath{\mathcal{S}}}
\newcommand{\Sww}{\Sw}
\newcommand{\ZZ}{\ensuremath{\mathbb{Z}}}

\newcommand{\dd}{\mathrm{d}}
\newcommand{\rr}{\boldsymbol{\rho}}
\newcommand{\ww}{\boldsymbol{\omega}}
\newcommand{\mm}{\boldsymbol{\mu}}
\newcommand{\supp}{\ensuremath{\mathrm{supp}\,}}
\newcommand{\EE}{\mathcal{E}}

\newcommand{\FF}{\mathcal{F}}

\newcommand{\GG}{\mathcal{G}}

\newcommand{\Bs}{\mathcal{B}_{\s}}
\newcommand{\Dww}{C^\infty_{c}}
\newcommand{\CC}{\mathbb{C}}

\newcommand{\s}{\mathfrak{s}}
\newcommand{\T}{\mathscr{T}}
\newcommand{\DD}{\mathscr{D}}
\newcommand{\PP}{\mathscr{B}_\s}

\newcommand{\para}{\varolessthan}

\newcommand{\mpara}{\mathord{\prec\!\!\!\prec}}

\newcommand{\B}{\mathfrak{B}}
\newcommand{\rA}{\mathfrak{A}}
\newcommand{\kk}{\Psi^{j}}
\newcommand{\KK}{\Psi^{<j-1}}
\newcommand{\KKu}{\KK_{x-u}}
\newcommand{\kkv}{\kk_{x-v}}
\newcommand{\cC}{\mathcal{C}}
\newcommand{\intr}{\int_{\RR^d}}
\newcommand{\dint}{\int \!\! \int}
\newcommand{\dintr}{\int_{\RR^d} \! \int_{\RR^d}}

\newcommand{\eps}{\varepsilon}

\newcommand{\uu}{\mathbf{1}}
\newcommand{\I}{\mathcal{I}}
\newcommand{\tT}{\mathcal{T}}
\newcommand{\V}{\mathcal{V}}
\newcommand{\W}{\mathcal{W}}

\newcommand{\AN}{A_{\NN}}
\newcommand{\nDl}[1]{\|#1\|_{\DD^\gamma(\mathbb{R}^d;\tT)}}
\newcommand{\nPl}[1]{\|#1\|_{\PP^\gamma(\mathbb{R}^d;\tT)}}

\newcommand{\A}{\mathscr{A}}

\newcommand{\Eww}{C^\infty}

\newcommand{\wexps}{\omega^{\mathrm{exp}}_\sigma}
\newcommand{\wpol}{\omega^{\mathrm{pol}}}

\makeatletter
\newcommand*\bigcdot{\mathpalette\bigcdot@{.5}}
\newcommand*\bigcdot@[2]{\mathbin{\vcenter{\hbox{\scalebox{#2}{$\m@th#1\bullet$}}}}}
\makeatother
\newcommand{\scl}{\bigcdot}
\newcommand{\plc}{\cdot}
\newcommand{\Warning}{{\color{red}\fontencoding{U}\fontfamily{futs}\selectfont\char 66\relax}}
\newcommand{\Span}[1]{\mathrm{span}\left\{#1\right\}}

\newcommand{\lsim}{\precsim}
\newcommand{\gsim}{\succsim}

\newcommand{\und}{\qquad \mbox{and} \qquad}
\newcommand{\oset}[2]{\overset{\tiny \mbox{#1}}{#2}}
\newcommand{\NNx}{\NN_{\times}}

\title{A Littlewood-Paley description of modelled distributions}
\author{Jörg Martin\thanks{Financial support by the DFG via Research Training Group RTG 1845 is gratefully acknowledged.} \\
  Institut f\"ur Mathematik \\
  Humboldt--Universit\"at zu Berlin \\
  \texttt{martin@math.hu-berlin.de}  
  \and
  Nicolas Perkowski\thanks{Financial support by the DFG via the Heisenberg Program and via Research Unit FOR 2402 is gratefully acknowledged.} \\
  Max-Planck-Institute for Mathematics in the Sciences, Leipzig \\
  \& Humboldt--Universit\"at zu Berlin \\
  \texttt{perkowsk@math.hu-berlin.de}}

\makeglossaries

\newglossaryentry{varphij}
{
	name={$\varphi_j$},
	description={Dyadic partition unity},
	sort={phij}
}

\newglossaryentry{varDeltaj}
{
	name={$\varDelta_j$},
	description={Littlewood-Paley block constructed from the dyadic partititon of unity \gls[noindex]{varphij}},
	sort={deltaj}
}

\newglossaryentry{kk}
{
	name ={$\kk$},
	description={Fourier transform of of the dyadic partition function \gls[noindex]{varphij}},
	sort={psij}
}

\newglossaryentry{KK}
{
	name={$\KK$},
	description={Abbreviation for $\sum_{i<j} \Psi^i$},
	sort={psijj}
}

\newglossaryentry{s}
{
	name={$\s$},
	description={Scaling vector},
	sort={s}
}

\newglossaryentry{ns}
{
	name={$\|\cdot\|_\s$},
	description={Scaled distance for some scaling \gls[noindex]{s}},
	sort={0s}
}

\newglossaryentry{GG}
{
	name={$\GG,\,\GG^\eps$},
	description={Bravais lattices, $\GG^\eps=\eps\cdot \GG$ denotes the scaled lattice},
	sort={g}
}

\newglossaryentry{widehatGG}
{
	name={$\widehat{\GG}$},
	description={Fourier cell for a Bravais lattice $\GG$},
	sort={g}
}

\newglossaryentry{isotropic}
{
	name={Isotropic},
	description={The scaling vector \gls[noindex]{s} equals $(1,\ldots,1)$},
	sort={isotropic}
}

\newglossaryentry{ww}
{
	name={$\ww$},
	description={Set of functions $\wpol,\,\wexps$ that classify weights},
	sort={o}
}

\newglossaryentry{rromega}
{
	name={$\rr(\omega)$},
	description={The set of weights, whose growth/decay is controlled by $\omega\in \mbox{\gls[noindex]{ww}}$},
	sort={ro}
}

\newglossaryentry{mathscrEeps}
{
	name={$\mathscr{E}^\eps$},
	description={Extension from Bravais lattices $\GG^\eps$ to $\RR^d$},
	sort={Eeps}
}

\newglossaryentry{diamond}
{
	name={$\diamond$},
	description={Symbol used for Wick products or Skohorod integration},
	sort={0}
}

\newglossaryentry{NNds}
{
	name={$|\NN^d|_\s$},
	description={Abbreviation for the set $\{|k|_\s\,\vert\,k\in \NN^d\}$ with the scaling vector \gls[noindex]{s} },
	sort={NNds}
}

\newglossaryentry{Bs01}
{
	name={$B_\s(0,1)$},
	description={Scaled unit ball for the scaling vector \gls[noindex]{s}},
	sort={Bs}
}

\newglossaryentry{Bgammapqs}
{
	name={$\B^\gamma_{p,q,\s}$},
	description={Besov space},
	sort={Bgpqs}
}

\newglossaryentry{cCgammas}
{
	name={$\cC^\gamma_\s$},
	description={Hölder-Zygmund space},
	sort={Cgs}
}

\newglossaryentry{Hgamma}
{
	name={$H^\gamma$},
	description={Fractional Sobolev space},
	sort={Hg}
}

\newglossaryentry{PFPi}
{
	name={$P(F,\Pi),\,P(F,\Gamma^\alpha)$},
	description={Paraproducts on a regularity structure},
	sort={PFP}
}

\newglossaryentry{DDgamma}
{
	name={$\DD^\gamma$},
	description={Space of modelled distributions},
	sort={Dg}
}

\newglossaryentry{DDetagamma}
{
	name={$\DD^{[\eta,\gamma]}$},
	description={Singular, modelled distributions},
	sort={Deg}
}

\newglossaryentry{A}
{
	name={$\A$},
	description={Green's function of a Fourier multiplier $a(D)$},
	sort={A}
}

\newglossaryentry{NN}
{
	name={$\NN$},
	description={Natural numbers including $0$, $\NN=\{0,1,2,\ldots\}$},
	sort={N}
}

\newglossaryentry{EEOmega}
{
	name={$\EE_{\Omega}$},
	description={Whitney extension for modelled distributions},
	sort={Eo}
}

\newglossaryentry{VbackslashW}
{
	name={$\V\backslash \W$},
	description={Complement of a sector},
	sort={VW}
}

\newglossaryentry{T}
{
	name={$\T=(A,\tT,G)$},
	description={Regularity structure},
	sort={T0}
}

\newglossaryentry{overlineT}
{
	name={$\overline{\T}=(\overline{A},\overline{\tT},\overline{G})$},
	description={Polynomial regularity structure},
	sort={T1}
}

\newglossaryentry{FF}
{
	name={$\FF$},
	description={Fourier transform with convention: $\FF f(x)=\int \dd \xi\, e^{2\pi \ii x \xi }\,f(\xi)$},
	sort={f}
}

\newglossaryentry{ast}
{
	name={$\ast$},
	description={Convolution on $\RR^d$},
	sort={0}
}

\newglossaryentry{D}
{
	name={$D$},
	description={Symbol used for Fourier multipliers on $\RR^d$},
	sort={D}
}

\newglossaryentry{AN}
{
	name={$\AN$},
	description={Shorthand for $\AN=A+|\NN^d|_\s$, where $A$ is the index set of a regularity structure \gls[noindex]{T}},
	sort={AN}
}

\newglossaryentry{PiGamma}
{
	name={$(\Pi,\Gamma)$},
	description={Model on a regularity structure \gls[noindex]{T}},
	sort={P}
}

\newglossaryentry{structurecondition}
{
	name={Structure condition},
	description={Condition on the components of a modelled distribution},
	sort={S}
}

\newglossaryentry{mmomega}
{
	name={$\mm(\omega)$},
	description={Set of jump measures for symmetric random walks},
	sort={mo}
}

\newglossaryentry{spectralsupport}
{
	name={Spectral support},
	description={The support of the Fourier transform of an (ultra-) distribution},
	sort={ss}
}

\newglossaryentry{varphilambda}
{
	name={$\varphi^\lambda$},
	description={$L^1$ scaling of $\varphi$},
	sort={pl}
}

\newglossaryentry{times}
{
	name={$\times$},
	description={Symbol used to connect products with factors in different lines},
	sort={0}
}

\newglossaryentry{Sww}
{
	name={$\Sww$},
	description={Ultra-differentiable Schwartz functions},
	sort={So}
}

\newglossaryentry{Eww}
{
	name={$\Eww$},
	description={Ultra-differentiable functions},
	sort={Co}
}

\newglossaryentry{Cnb}
{
	name={$C^n_b$},
	description={Functions with bounded derivatives up to order $n$},
	sort={Cnb}
}

\newglossaryentry{lesssim}
{
	name={$\lesssim$},
	description={Means $\leq$ ``up to a multiplicative, deterministic constant''},
	sort={0l}
}

\newglossaryentry{lsim}
{
	name={$\lsim$},
	description={Used for indices $i,j\in \ZZ$. Means $\leq$ ``up to an additive, deterministic constant''},
	sort={0l}
}

\newglossaryentry{Rgammaxh}
{
	name={$R^\gamma_{x;h}$},
	description={(Anisotropic) Taylor remainder, $R^\gamma_{x;h}=\mathrm{Id}-\mbox{\gls[noindex]{Tgammaxh}}$},
	sort={Rgxh}
}

\newglossaryentry{Tgammaxh}
{
	name={$T^\gamma_{x;h}$},
	description={(Anisotropic) Taylor polynomial up to order $\gamma$},
	sort={Tgxh}
}

\newglossaryentry{NNdgammal}
{
	name={$\NN^d_{<\gamma}$},
	description={Set of $k\in \NN^d$ with $|k|_\s<\gamma$},
	sort={Ndg0}
}

\newglossaryentry{NNdgammag}
{
	name={$\NN^d_{>\gamma}$},
	description={``Boundary'' of $\NN^d_{<\gamma}$},
	sort={Ndg1}
}

\newglossaryentry{mathfrakmk}
{
	name={$\mathfrak{m}(k)$},
	description={Defined for $k\in \NN^d\backslash\{0\}$ as $\mathfrak{m}(k)=\min\{j\,  \vert \,k_j\neq 0\}$},
	sort={m}
}

\newglossaryentry{LpRRdrho}
{
	name={$L^p(\RR^d,\rho)$},
	description={Weighted $L^p$ space with the convention $\|f\|_{L^p(\RR^d,\rho)}=\|f\cdot \rho\|_{L^p(\RR^d)}$},
	sort={Lp}
}

\newglossaryentry{Gammaalphatau}
{
	name={$\Gamma^\alpha \tau$},
	description={Projection of $\Gamma \tau$ onto $\tT_\alpha$ on some regularity structure \gls[noindex]{T}},
	sort={Gat}
}

\newglossaryentry{taualpha}
{
	name={$\tau^\alpha,\,\tau^{<\gamma}$},
	description={Projection of $\tau\in \tT$ on subspaces for some regularity structure $\T=(A,\tT,G)$},
	sort={ta}
}

\newglossaryentry{functionlike}
{
	name={function like},
	description={Attribute for a sector of non-negative regularity},
	sort={functionlike}
}

\newglossaryentry{tauuu}
{
	name={$\tau^{\uu}$},
	description={Coefficient of $\tau\in \tT$ in front of $\uu$},
	sort={tau1}
}

\newglossaryentry{OmegaT}
{
	name={$\Omega^T$},
	description={Denoting the set $(0,T)\times \RR^{d-1}$},
	sort={OT0}
}

\newglossaryentry{OmegaTt}
{
	name={$\Omega^T_t$},
	description={Denoting the set $(t,T)\times \RR^{d-1}$ for $t\in (0,T)$},
	sort={OT1}
}

\begin{document}

\maketitle

\begin{abstract}
We exhibit a fundamental link between Hairer's theory of regularity structures \cite{RegularityStructures} and the paracontrolled calculus of \cite{GIP}. By using paraproducts we provide a Littlewood-Paley description of the spaces of modelled distributions in regularity structures that is similar to the Besov description of classical Hölder spaces. 
\end{abstract}


\section{Introduction}

This article builds a bridge between two different approaches that arose from the study of singular stochastic partial differential equations (singular SPDEs). These equations are ill posed because of the interplay of highly irregular noise with nonlinearities, which often leads to resonances that may have to be removed by a renormalization procedure. A new way of describing distributions, their regularity, and operations on them was required in order to give a meaningful solution theory for singular SPDEs, and this was implemented differently by different groups. Our aim is to reveal a link between  \emph{regularity structures}, which were developed by Hairer in \cite{RegularityStructures} and applied abundantly since (for example \cite{HairerQuastel, HairerLabbeR3, Hairer2015Central, HP15, Cannizzaro2017Malliavin, MatetskiHairer}), and the Fourier approach of \emph{paracontrolled distributions}, which first appeared in the work \cite{GIP} by Gubinelli, Imkeller and Perkowski and was used for instance in \cite{CatellierChouk, CannizzaroChouk, WeberMourratPhi43, KPZreloaded, ChoukAllez, Zhu2015, DiscreteParacontrolled}. These two techniques are by no means the only tools for  singular SPDEs and alternative views were provided by Kupiainen in his application of renormalization group techniques \cite{KupiainenPhi43, KupiainenKPZ}, and by Otto and Weber with their rough path flavored approach~\cite{Otto2016, Otto2018}.

In this paper however we will only focus on the comparison of two seemingly distinct notions from regularity structures and paracontrolled distributions. Central to the theory of \cite{RegularityStructures} is the concept of a \emph{modelled distribution}. These are generalized H\"older functions that are classified by spaces called $\DD^\gamma(\RR^d;\tT)$ with $\gamma\in \RR$. The sets $\DD^\gamma(\RR^d;\tT)$ collect functions $F\colon\RR^d\rightarrow \tT$ that take values in a graded vector space $\tT=\bigoplus_{\alpha\in A} \tT_\alpha$ (with $A\subseteq \RR$ and with $\tT_\alpha$ being normed spaces) and whose components $F^\alpha\colon\RR^d\rightarrow \tT_\alpha$ satisfy for $x,\,y\in \RR^d$ \footnote{Instead of $|y-x|$ we will consider a scaled distance 
$\|y-x\|_{\s}$ below, as in \cite{RegularityStructures}.} 
\begin{align}
\label{eq:IntroductionModelledDistribution}
\|F^\alpha_y-\Gamma_{yx}^\alpha F_x \|_{\tT_\alpha} \lesssim |y-x|^{\gamma-\alpha}
\end{align} 
for $\alpha<\gamma$.  The object $\Gamma_{yx}^\alpha F_x \in \tT_\alpha$ in \eqref{eq:IntroductionModelledDistribution}  should be thought of as a local ``Taylor-like'' expansion. 
The most basic case arises when we consider the so called polynomial regularity structure $\overline{\tT}$, where $F^\alpha$ just represents the derivative of order $\alpha$ of a function $f$ and 
$\Gamma_{yx}^\alpha F_x$ incorporates the Taylor expansion of this derivative, so that $\mathscr{D}^\gamma(\RR^d;\overline{\tT})$ just represents the  $\gamma$-Hölder continuous functions (for a precise formulation see Lemma~\ref{lem:HoelderAsModelled} below). It is rather classical that the space of $\gamma$-Hölder continuous functions can also be described via Littlewood-Paley theory as a Besov space $\mathcal{B}^\gamma(\RR^d)$, so that there are in fact various ``classical'' descriptions of $\mathscr{D}^\gamma(\RR^d;\overline{\tT})$. However, in general the space $\mathscr{D}^\gamma(\RR^d,\tT)$ can be of a far more complex structure and usually it involves expansions in non-polynomial objects, which are typically constructed from an underlying SPDE. In such cases there is no description of $\mathscr{D}^\gamma(\RR^d;\tT)$ in terms of more familiar function spaces. The central role of modelled distributions in the theory of regularity structure is due to the so called reconstruction theorem \cite[Theorem 3.10]{RegularityStructures} which shows that there is a (unique) distribution $f\in \Sw'(\RR^d)$ that is described by the modelled distribution $F\in \DD^{\gamma}(\RR^d;\tT)$. On the level of $\DD^\gamma(\RR^d;\tT)$ there is a robust theory for operations such as multiplication whose execution is typically hazardous on $\Sw'(\RR^d)$, and describing the objects in an SPDE via modelled distributions leads to a robust formulation of the Schauder theory. But since there is no description of $\mathscr{D}^\gamma(\RR^d;\tT)$ in terms of classical function spaces, all these results have to be derived from scratch and are typically quite cumbersome to prove.

The linchpin in the paracontrolled framework is the notion of a \emph{paracontrolled distribution}. Roughly speaking, a distribution $f\in \Sw'(\RR^d)$, often the solution to an SPDE, is called paracontrolled if it can be smoothened by the subtraction of a \emph{paraproduct}, a sort of ``frequency modulation'' of a given reference distribution. This allows us to transform the considered SPDE into an equation that can be solved via classical Schauder theory. Although the paracontrolled approach might seem quite different from the one in regularity structures, note that an increase of smoothness could be interpreted as the cancellation of fluctuations on small scales. From this point of view the paraproduct seems to capture the local behaviour of the considered distribution. This philosophy seems reminiscent to the idea of a local expansion in \eqref{eq:IntroductionModelledDistribution}. Moreover, often similar objects appear in the expansion $\Gamma^\alpha_{yx}F_x$ in \eqref{eq:IntroductionModelledDistribution} and the paraproducts used to control the solutions to SPDEs: Compare for instance the solution theory for the parabolic Anderson model \cite{RegularityStructures, GIP}, the $\Phi^4_3$ model \cite{RegularityStructures, CatellierChouk} or the KPZ equation \cite{FrizHairer, KPZreloaded}. Based on this similarity it has been conjectured in \cite{GIP} that there might be a one-to-one correspondence between modelled distributions and paracontrolled distributions. Here we give a precise formulation of that conjecture and we prove it.

More precisely, we propose the following description of a modelled distribution via paraproducts: In Definition \ref{def:Paraproducts} we define the paraproducts $P(F,\Gamma^\alpha)$ for $\alpha \in A$ and $F\colon\RR^d\rightarrow \tT=\bigoplus_{\alpha\in A} \tT_\alpha$. We then say that $F$ is in the space $\mathscr{B}^\gamma(\RR^d;\tT)$ if for $\alpha<\gamma$
\begin{align}
\label{eq:IntroductionPgamma}
F^\alpha-P(F,\Gamma^\alpha)\in \mathcal{B}^{\gamma-\alpha}\,, 
\end{align}
where $\mathcal{B}^{\gamma-\alpha} = B^{\gamma-\alpha}_{\infty,\infty}$ denotes a Besov space, and if further a certain \emph{structure condition} (Definition \ref{def:StructureCondition} below) is satisfied. In Theorem \ref{thm:BridgeTheorem} we then show that
\begin{theorem*}
For $\gamma\in\RR\backslash (\NN+A)$ we have
\begin{align}
\label{eq:IntroductionBridge}
\DD^\gamma(\RR^d;\tT)=\mathscr{B}^\gamma(\RR^d;\tT)\,
\end{align}
with equivalent norms.
\end{theorem*}
One might consider $\mathscr{B}^\gamma(\RR^d;\tT)$ as a type of Besov space on $\tT$. Indeed, in the case of the polynomial regularity structure $\tT=\overline{\tT}$ it turns out that $\mathscr{B}^\gamma(\RR^d;\tT)$ just collapses to the classical Besov space $\mathcal{B}^\gamma$, so that \eqref{eq:IntroductionBridge} can be interpreted as a generalization of the Besov description of Hölder continuous functions. Since both, $\mathcal{B}^\gamma$ and $P(F,\Gamma^\alpha)$, will be built using Littlewood-Paley theory, \eqref{eq:IntroductionBridge} is a Littlewood-Paley description of $\mathscr{D}^\gamma(\RR^d;\tT)$ and moreover a characterization in terms of more classical spaces from analysis.

Although we consider the theorem above an interesting observation in its own right, one should also judge its importance from a more practical point of view. For many operators, such as Fourier multipliers, it is more natural to work with the spectral description of function spaces provided by the Littlewood-Paley theory. The theorem above therefore yields a first insight on how to describe modelled distributions in this formulation. However, an elaboration of a full, say, Schauder-like statement based on \eqref{eq:IntroductionBridge} would still require a lot of technical work, which is why we refrain from doing so in this paper. For ideas and results in that direction we refer to the Ph.D. thesis of the first author \cite{Dissertation}.

 \subsubsection*{Structure of this article}
 
Section \ref{sec:background} recalls the fundamental concepts we will use. These are mostly some definitions and basic facts concerning Besov spaces and some elementary definitions from the theory of regularity structures. Since statements about Besov regularity, such as \eqref{eq:IntroductionPgamma}, actually concern objects with values in Banach spaces, we will also repeat the essentials of vector-valued distributions.

Conceptually new  definitions are given in Section \ref{sec:paramodelled}, where we introduce the paraproducts $P(F,\Gamma^\alpha)$ and the space $\mathscr{B}^\gamma(\RR^d;\tT)$ mentioned above. 

The main result of this paper, that is identity \eqref{eq:IntroductionBridge}, is presented and proved in Section \ref{sec:MainTheorem}. 

\subsection*{Notation}
We will write $x\lesssim y$
to denote an inequality of the form $x\leq C \cdot y$ for some constant $C>0$, which is independent of $x$ and $y$. We also use the notation $x\gtrsim y$ to describe the relation $y\lesssim x$. We write $x\approx y$ to indicate that both statements, $x\lesssim y$ and $x\gtrsim y$, are true.

For indices $i,j\in \ZZ$ we will write
\glsadd{lsim}
\begin{align}
\label{eq:lsim}
i\lsim j \qquad &\Leftrightarrow i \leq j+N
\end{align}
where $N\in \ZZ$ is a deterministic constant which is independent of $i$ and $j$, i.e. if $2^i \lsim 2^j$. We also use $i\gsim j$ to denote $i\lsim j$ and $i \sim j$ whenever $i \lsim j$ and $j\lsim i$. The symbol $\NN=\{0,1,2,\ldots\}$ denotes in this article the natural numbers including $0$. The set $\NNx:=\NN\backslash\{0\}=\{1,2,3,\ldots\}$ is then defined by exclusion of the latter.  

We use the symbol $C^r(\RR^d)$ for the set of $r\in \NN$ times differentiable, complex-valued functions and set $C^\infty(\RR^d)=\bigcap_{r\in \NN} C^r(\RR^d)$. To indicate bounded derivates we add an index $b$, so that $C_b(\RR^d)$ denotes for instance the set of bounded, continuous and complex-valued functions, equipped with the norm
\begin{align*}
\|f\|_{C_b(\RR^d)}=\sup_{x\in \RR^d}|f(x)|\,.
\end{align*}
$C^\infty_c(\RR^d)$ is used for the subset of functions in $C^\infty(\RR^d)$ that have compact support. $\Sw(\RR^d)$ denotes the Schwartz space. A codomain of a function space which is different from the complex numbers will be mentioned in the notation via a semicolon, so that $C^\infty(\RR^d;[0,1])$ denotes the smooth functions on $\RR^d$ which take values in the unit interval and $C^\infty_c(\RR^d;X)$ is the set of smooth, compactly supported functions with values in a Banach space $X$. Banach-valued distributions $\Sw'(\RR^d;X)$ and Besov spaces $\Bs^\gamma(\RR^d;X)$ will be introduced in Subsection~\ref{subsec:Besov}.

Given a function $f$ on $\RR^d$ we will occasionally write shorthand $f_x:=f(x)$ for $x\in \RR^d$ and use the ``physics convention'' in writing the differential right after the integral sign, for instance
\begin{align*}
\int_{\RR^d}\dd x\, f_x\,.
\end{align*}

We follow \cite{RegularityStructures} in the definition of anisotropic distances and scalings:  Given a \emph{scaling vector}
\glsadd{s}
\begin{equation}
\label{eq:ScalingVector}
\s\in \NNx^d\,,
\end{equation}
we define for $x\in \RR^d$ the \emph{scaled ``norm''}
\begin{align}
\label{eq:AnisotropicDistance}
\|x\|_\s:=\sum_{i=1}^d |x_i|^{1/\s_i}\,,
\end{align}
and we define the scaled unit ball as
\glsadd{Bs01}
\begin{align*}
B_\s(0,1)=\{x\in \RR^d\,\vert\, \|x\|_\s <1\}\,.
\end{align*}

For multi-indices $k\in \NN^d$ we set
\[
	|k|_\s=\sum_{i=1}^d k_i \s_i\,,\qquad |\s|=\sum_{i=1}^d \s_i\,.
\]
Finally we write for positive $a>0$
\begin{align}
\label{eq:as}
a^{\s}:=\mathrm{diag}(a^{\s_1},\,a^{\s_2},\, \ldots,a^{\s_d})\,,
\end{align}
where $\mathrm{diag}(\,\cdot\,)$ denotes the diagonal matrix with diagonal ``$\,\cdot\,$''. For $\varphi\in L^1(\RR^d)$ we sometimes need the notation $\varphi^\lambda$ by which we mean the $L^1$-scaling
\glsadd{varphilambda}
\begin{align*}
\varphi^\lambda:=\lambda^{-|\s|}\varphi(\lambda^{-\s}\plc)\,,
\end{align*}
so that $\varphi^\lambda_{x}$ should be read as 
\begin{align}
\label{eq:Scaledvarphi}
\varphi^\lambda_{x}=\lambda^{-|\s|}\varphi(\lambda^{-\s}x)\,,
\end{align}
where the \emph{matrix} $\lambda^{-\s}$ is defined as in \ref{eq:as}. 
\begin{remark} \Warning  Note that we slightly differ here from \cite{RegularityStructures}, where this notation denotes the \emph{function} $z\mapsto \lambda^{-|\s|}\varphi(\lambda^{-\s}(z-x))$ instead. For us $\varphi^\lambda_x$ denotes a complex number and not a function. 
\end{remark}

\section{Background}
\label{sec:background}

In this section we recall the basic definitions we will make use of in this article. In Subsection~\ref{subsec:Besov} we repeat the definition and basic properties of Besov spaces. We need two slightly unusual modifications which is the allowance for \emph{anisotropic} smoothness as well as the possibility for the considered distributions to take values in a Banach space. We therefore start with a short repetition concerning vector-valued distributions. 

In Subsection \ref{subsec:RegularityStructures} we recall a few elementary definitions from the theory of regularity structures. Essential to us is Definition~\ref{def:Dgamma} which introduces the space $\DD^\gamma(\RR^d;\tT)$ of modelled distributions. 

\subsection{Anisotropic Besov spaces}
\label{subsec:Besov}

\subsubsection*{Banach-valued distributions}
It will be convenient for us to work with vector-valued distributions in this article. Although this theory is rather classical in the literature (see e.g. \cite{SchmeisserTriebel, BanachValuedDistributions} or even \cite{Schwartz}), the scalar-valued case is usually a more familiar concept, so that we here briefly repeat the most important aspects of the vector modification of distribution theory.  

For a Banach space $X$ we define the \emph{Banach-valued} Schwartz distributions $\Sw'(\RR^d;X)$ as the  set of continuous linear functionals
\begin{align*}
f:\,\Sw(\RR^d)\rightarrow X\,,
\end{align*}
where $\Sw(\RR^d)$ is just the classical space of complex-valued Schwartz functions. A measurable functions $f:\RR^d\rightarrow X$ such that
\begin{align}
\label{eq:BochnerRepresentation}
 f(\varphi):=\intr \dd x \,f(x)\varphi(x)  
 \end{align} 
 is well-defined for any $\varphi \in \Sw(\RR^d)$ as a Bochner integral can then be identified via \eqref{eq:BochnerRepresentation} with a distribution in $\Sw'(\RR^d;X)$. The elements of $\Sw'(\RR^d;X)$ can therefore be interpreted as generalized functions with values in $X$. Taking $X=\mathbb{C}$ we simply have
\begin{align*}
 \Sw'(\RR^d;\mathbb{C})=\Sw'(\RR^d)\,,
\end{align*}
where $\Sw'(\RR^d)$ is the usual space of tempered distributions. 

 Most concepts which are known for $X=\mathbb{C}$ carry over to the general case. We can for instance define addition and multiplication by constants via
 \begin{align*}
 (f+\tilde{f})(\varphi):=f(\varphi)+\tilde{f}(\varphi),\qquad (c\cdot f)(\varphi):=f(c\cdot \varphi)\,.
 \end{align*}
 for $f,\tilde{f}\in \Sw'(\RR^d;X),\,\varphi\in \Sw(\RR^d)$ and $c\in \CC$. 
 Multiplication with a smooth $\psi\in C^\infty(\RR^d)=C^\infty(\RR^d;\CC)$ whose derivatives grow at most polynomially can further be defined as
 \begin{align*}
 (\psi\cdot f)(\varphi):=f(\psi\cdot \varphi)\,.
 \end{align*}
The support of a distribution $f\in \Sw'(\RR^d;X)$ is defined as
 \begin{align*}
 \supp f := \overline{\{x\in \RR^d\,\vert\, \forall r>0 \,\exists \varphi\in C^\infty_c(B(x,r))\,\text{s.t.}\,f(\varphi)\neq 0 \}}\,.
 \end{align*} 
 There is also a Fourier transform $\FF,\,\FF^{-1}$ given on $\Sw'(\RR^d;X)$, defined via
 \begin{align*}
 \FF f(\varphi):=f(\FF \varphi) \und \FF^{-1} f(\varphi):=f(\FF^{-1}\varphi)
 \end{align*}
 for $f\in \Sw'(\RR^d;X)$ and $\varphi\in \Sw(\RR^d)$ and where we use the convention
 \begin{align*}
 \FF \varphi(x)=\intr \dd \xi\, e^{-2\pi \ii x\scl \xi} \varphi(\xi) \und
\FF^{-1} \varphi(x)=\intr \dd \xi\, e^{2\pi \ii x\scl \xi} \varphi(\xi)\,,
 \end{align*}
 where $\scl$ denotes the scalar product on $\RR^d$. As in the case $X=\CC$ we see that $\FF$ and $\FF^{-1}$ define continuous bijections on $\Sw'(\RR^d;X)$ that are inverse to each other.  
As in the scalar case, that is $X=\CC$, the Fourier transform of every compactly supported $f\in \Sw'(\RR^d;X)$ is smooth. We further have the usual relations for convolution and multiplication between say $f\in \Sw'(\RR^d;X)$ and $\psi\in \Sw(\RR^d)=\Sw(\RR^d;\CC)$:
\begin{align*}
\FF(\psi\ast f)=\FF\psi \cdot \FF f,\qquad \FF^{-1}(\psi\ast f)=\FF^{-1}\psi \cdot \FF^{-1} f\,,\\
\FF(\psi \cdot f)=\FF \psi \ast \FF f,\qquad \FF^{-1}(\psi\cdot f)=\FF^{-1}\psi \ast \FF^{-1} f\,,
\end{align*}
where convolutions such as $\psi\ast f$ are defined as 
\begin{align*}
(\psi \ast f)(\varphi):=f(\psi\ast \varphi_{-\cdot})\,.
\end{align*}
Notice that in contrast to the case $X=\CC$, in general there is no meaningful convolution concept between two distributions in $\Sw'(\RR^d;X)$ as $X$ is not equipped with a product. 

We use the term \emph{spectral support} to denote the support of the Fourier transform of an $X$-valued distribution. 

\subsubsection*{Besov spaces}

The following types of sets will be the building blocks for our definitions below. 

\begin{definition}
\label{def:RectangularAnnulus}
We say that a set $\B\subseteq \RR^d$ is a box if there are $a_1,\ldots,a_d> 0$ such that $\B=\bigtimes_{i=1}^d [-a_i,a_i]$. A set $\rA\subseteq \RR^d$ is a rectangular annulus if there are two boxes $\B,\,\tilde{\B}\subseteq \RR^d$ with $ \B\subseteq\tilde{\B}$ and $\partial \B\cap \partial \tilde{\B}=\varnothing$ such that $\rA=\overline{\tilde{\B}\backslash \B} $.
\end{definition}
We then have the following elementary, geometric properties.
\begin{lemma}
\label{lem:FinitelyManyAnnuli}
Let $\rA,\tilde{\rA}$ be two rectangular annuli, let $\B$ be a box and let $\s\in \NNx^d$ be a scaling vector. If we define for $j\geq 0$ the scaled sets $\rA_j:=2^{j\s} \rA=\{2^{j\s}x\,\vert\,x\in \rA\},\,\tilde{\rA}_j:=2^{j\s} \tilde{\rA}$ and $\B_j:=2^{j\s }\B$ (with a \emph{matrix} $2^{j\s}$ as in \eqref{eq:as}), we have the following relations:
\begin{itemize}
\item $\rA_i\cap \B_j\neq \varnothing$ only if $i\lsim j$,
\item $\rA_i\cap \tilde{\rA}_j\neq \varnothing$ only if $i\sim j$,
\end{itemize}
where $\lsim$ and $\sim$ should be read as on page \pageref{eq:lsim}. 
\end{lemma}
\begin{proof}
For the first statement we can write $2^{i\s} \rA \cap 2^{j\s}\B\neq \varnothing$ as $\rA\cap 2^{(j-i)\s}\B\neq \varnothing$ and use that for $i\gsim j$ this cannot be true by definition of a rectangular annulus. The second statement then follows from the first one if we use that $\rA$ and $\tilde{\rA}$ are each contained in some box. 
\end{proof}

We now construct an \emph{anisotropic dyadic partition of unity} for wich we essentially follow \cite[Section 5.1]{TriebelIII}. Fix in the following some scaling vector $\s\in \NNx^d$. \footnote{To be precise, \cite{TriebelIII} works with $\s\in (0,\infty)^d,\,\sum_{i=1}^d \s_i=d$ instead, which allows for an interpretation of the Besov regularity in Definition \ref{def:Besov} below as a sort of \emph{mean} regularity. We here follow the scaling that corresponds to the definitions in \cite{RegularityStructures}.}

Let $\B_{-2}:=\varnothing$ and set $\B_j:=2^{(j+1)\s}[-1,1]^d$ for $j\geq -1$. Fix further a symmetric and positive $\varphi_{-1} \in \Dww(\RR^d)$ with values in $[0,1]$ such that $\varphi_{-1}=1$ on an open set containing $\B_{-1}=[-1,1]^d$ and $\supp \varphi_{-1}\subseteq \B_0=2^{\s}[-1,1]^d$. We then set for $j\geq 0$ \glsadd{varphij} \label{DyadicPartitionOfUnity}
\begin{align*}
\varphi_{j}=\varphi_{-1}(2^{-(j+1)\s}\cdot)-\varphi_{-1}(2^{-j\s}\cdot)\,,
\end{align*}
(with \emph{matrices} $2^{-j\s},\,2^{-(j+1)\s}$ as in \eqref{eq:as}). This yields a family $(\varphi_j)_{j\geq -1}\in \Dww(\RR^d)$ that satisfies the following properties:
\begin{itemize}
\item $\varphi_j(x)\geq 0$ for $j\geq -1$ and $x\in \RR^d$.
\item $\sum_{j \geq -1} \varphi_j(x) =1$ for $x\in \RR^d$.
\item $\supp\,\varphi_{j}\subseteq \B_{j+1}\backslash \B_{j-1}$ for $j\geq -1$, in particular $\supp\,\varphi_j \cap \supp \varphi_{j'}=\varnothing$ for $|j-j'|>1$,
\end{itemize}
and that further exhibits the following scaling behavior for $j\geq 0$
\begin{align}
\label{eq:Scalingphij}
\varphi_j &=\varphi_0(2^{-j\s}\cdot)\,,\\
\label{eq:ScalingSum}
\sum_{i<j} \varphi_i &=\varphi_{-1}(2^{-j\s}\cdot)\,.
\end{align}
The family $(\varphi_j)_{j\geq -1}$ is called an \emph{(anisotropic) dyadic partition of unity}. 
Note that for $j\geq 0$ the support of $\varphi_j$ is contained in a rectangular annulus of the form $2^j \rA$ as in Definition~\ref{def:RectangularAnnulus} (namely the set $\overline{\B_{j+1}\backslash \B_{j-1}}$). The support of $\varphi_{-1}$ is contained in the box $\B_0$.

\begin{remark}
\label{rem:OtherPartitionOfUnity}
The choice of the sequence $\B_j$ was of course rather arbitrary. One could for example have chosen another sequence of boxes such as $\tilde{\B}_j=a\cdot\B_j$ with some $a>0$ instead, which would lead to a different partition of unity $(\tilde{\varphi}_j)_{j\geq -1}$. 
\end{remark}

The following functions will play a special role in this paper: \glsadd{kk} \glsadd{KK}
\begin{align}
\label{eq:Psij}
\Psi^j:=\FF^{-1}\varphi_j\,, \qquad \Psi^{<j}=\sum_{-1\leq i<j} \Psi^i=\FF^{-1} \left(\varphi_{-1}(2^{-j\s}\cdot )\right)
\end{align}
with $j\geq -1$ (note that $\Psi^{<-1}=0$). We also use occasionally the notation $\Psi^{\leq j}:=\Psi^{<j+1}$ for $j\geq -1$.


\begin{lemma}[Scaling]
\label{lem:PsiScal}
There are real-valued, symmetric Schwartz functions $\phi_1,\,\phi_2\in \Sw(\RR^d)$ such that for $j\geq 0$
\begin{align*}
\kk=2^{j|\s|}\,\phi_1(2^{j\s}\cdot)\mbox{\qquad and \qquad} \KK= 2^{j|\s|} \,\phi_2(2^{j\s}\cdot)\,,
\end{align*}
where the \emph{matrix} $2^{j\s}$ should be read as in \eqref{eq:as}. In particular we have for $k\in \NN^d$
\begin{align*}
\|\partial^k\Psi^j\|_{L^1(\RR^d)},\,\|\partial^k \Psi^{<j}\|_{L^1(\RR^d)}\lesssim 2^{j|k|_\s}\,.
\end{align*}
\end{lemma}
\begin{proof}
The scaling property follows from \eqref{eq:Scalingphij} and \eqref{eq:ScalingSum} via $\phi_1:=\FF^{-1}\varphi_0$ and $\phi_2:=\FF^{-1} \varphi_{-1}$. Since $\varphi_0$ and $\varphi_{-1}$ are symmetric and real-valued, also their Fourier transforms $\phi_1$ and $\phi_2$ are symmetric and real-valued.
\end{proof}

The action of $\Psi^{<j}$ on polynomials can be described as follows. 
\begin{lemma}[Interplay with polynomials]
\label{lem:PsiPol}
For $k,\,l\in \NN^d$ and $j\geq 0$ we have
\begin{align*}
\intr \dd u \,\partial^k \Psi^{j}_{u} \,u^l=0 
\mbox{\qquad and\qquad } \intr \dd u \,\partial^k \Psi^{<j}_{u} \,u^l=\delta_{k l}\,k! \,.
\end{align*}
\end{lemma}
\begin{proof}
Note that we can replace $\Psi^{j}_{u}$ and $\Psi^{<j}_{u}$ in the expressions above by $\Psi^{j}_{-u}$ and $\Psi^{<j}_{-u}$ due to symmetry. We consider the right equality first. By integration by parts the left hand side of this relation equals $\binom{l}{k} k! \intr\dd u\,\Psi^{<j}_{-u}\,u^{l-k}$ (or $0$ if $l>k$). This expression can then be rewritten via the inverse Fourier transform as 
\begin{align*}
\uu_{l\leq k} \binom{l}{k} k!\,(2\pi \ii)^{k-l}\partial^{l-k} \FF^{-1}\Psi^{<j}_{-\cdot}(0)\,,
\end{align*}
which yields the claim since $\FF^{-1} \Psi^{<j}$ equals $1$ in a box around $0$. The left relation in the claim is shown in the same way by using $\FF^{-1}\Psi^{j}_{-\cdot}(0)=\varphi_j(0)=0$ instead.
\end{proof}

\begin{lemma}[Interplay with polynomially growing functions]
\label{lem:PsiPolScal}
If for some $a\geq 0$ and some measurable $f:\RR^d\rightarrow \mathbb{C}$ we have $|f_x| \lesssim \|x\|_\s^a$, then this implies for $k\in \NN^d$
\begin{align*}
\left|\intr \dd u \, \partial^k\kk_u \cdot f_u \right| + \left|\intr\, \dd u \, \partial^k\Psi^{<j}_u \cdot f_u \right|\lesssim 2^{-j(a-|k|_\s)}\,.
\end{align*}
\end{lemma}
\begin{proof}
This follows immediately from Lemma \ref{lem:PsiScal}.
\end{proof}

Let $X$ be a Banach space. Using the functions $\varphi_j$ or their inverse Fourier transforms $\kk=\FF^{-1}\varphi_j$ we define \emph{Littlewood-Paley blocks} for $X$-valued distributions $f\in \Sww'(\RR^d;X)$ and $j\geq -1$ by
\glsadd{varDeltaj} 
\begin{align}
\varDelta_j f=\FF^{-1}( \varphi_j \cdot \FF f)=\Psi^j\ast f=\intr\dd u\, \kk_{\cdot-u} \,f_u\,,
\label{eq:LittlewoodPaleyBlock}
\end{align}
where multiplication and convolution are defined as in the beginning of this subsection and where we used formal notation in the integral on the right hand side. Note that $\varDelta_j f$ is in $C^\infty(\RR^d;X)$ as the inverse transform of a compactly supported distribution. We can now decompose any $f\in \Sww'(\RR^d;X)$ by its \emph{Littlewood-Paley decomposition}:
\begin{align}
f=\sum_{j\geq -1} \varDelta_j f \,,
\label{eq:LittlewoodPaleyDecomposition}
\end{align}
where the sum on the right hand side converges in the topology of $\Sw'(\RR^d;X)$. Note that we have for $p\in [1,\infty]$
\begin{align}
\label{eq:LpLp}
\|\varDelta_j f\|_{L^p(\RR^d;X)} =\|\kk\ast f\|_{L^p(\RR^d;X)}\lesssim \|f\|_{L^p(\RR^d);X)}\,,
\end{align}
where the involved constant can be chosen independently of $j$. Indeed, Young's convolution inequality (which still holds if one of the factors is $X$-valued) implies for $f\in L^p(\RR^d;X)$
\begin{align*}
\|\varDelta_j f\|_{L^p(\RR^d;X)} &=\|\kk\ast f\|_{L^p(\RR^d;X)} \leq \|\kk\|_{L^1(\RR^d)} \|f\|_{L^p(\RR^d;X)} \oset{Lemma \ref{lem:PsiScal}}{\lesssim} \|f\|_{L^p(\RR^d;X)}\,,
\end{align*}
where we used that due to Lemma \ref{lem:PsiScal} we have $\|\kk\|_{L^1(\RR^d)}\|=\|\phi_1\|_{L^1(\RR^d)}\lesssim 1$ for $j>-1$.
The same argument shows that 
\begin{align*}
S_{j} f:=\sum_{-1\leq i<j} \Delta_i f=\Psi^{<j}\ast f=\intr \dd u \,\Psi^{<j}_{\cdot -u} \,f_u\,,
\end{align*}
is also a bounded operator from $L^p(\RR^d;X)$ to itself for $p\in [1,\infty]$. 
Using the decomposition \eqref{eq:LittlewoodPaleyDecomposition} we can now define anisotropic, $X$-valued Besov spaces.
\begin{definition}
\label{def:Besov}
Let $\gamma\in \RR$, let $X$ be a Banach space and let $\s\in \NNx^d$ be a scaling vector. Let further $(\varphi_j)_{j\geq -1}$ be a dyadic partition of unity on $\RR^d$ defined as above and constructed with $\s$. The \emph{anisotropic Besov space} $\Bs^\gamma(\RR^d;X)$ is given by 
\begin{align*}
	\Bs^\gamma(\RR^d;X):=\left\{f\in \Sw'(\RR^d;X)\,\middle\vert \, \|f\|_{\Bs^\gamma(\RR^d;X)}<\infty\right\}\,,
\end{align*}
where 
\begin{align}
\label{eq:BesovNorm}
\|f\|_{\Bs^{\gamma}(\RR^d;X)}:=\left\| \left( 2^{j\gamma } \|\Delta_j f\|_{L^\infty(\RR^d;X)}\right)_{j\geq -1} \right\|_{\ell^\infty}\,.
\end{align}
with the Littlewood-Paley blocks $(\varDelta_j)_{j\geq -1}$ defined via $(\varphi_j)_{j\geq -1}$ as in \eqref{eq:LittlewoodPaleyBlock}.  
\end{definition}

\begin{remark}
\label{rem:BesovGeneralIntegrability}
 The norm \eqref{eq:BesovNorm} can also be defined with general $L^p(\RR^d;X)$ and $\ell^q$ norms for $p,\,q\in [1,\infty]$ (compare \cite[Section~2.7]{Bahouri}) which gives rise to a more general space $\mathcal{B}^\gamma_{p,q,\s}(\RR^d;X)$ that also accounts for the integrability of the considered objects. For the sake of simplicity we will only consider the case $p=q=\infty$ here.    
\end{remark}
\begin{remark}
Using Lemma \ref{lem:FinitelyManyAnnuli} one sees that another choice of dyadic partition of unity $(\tilde{\varphi}_j)_{j\geq -1}\subseteq \Dww(\RR^d)$ instead of $(\varphi_j)_{j\geq -1}$, as in Remark \ref{rem:OtherPartitionOfUnity}, gives an equivalent norm for $\Bs^\gamma(\RR^d;X)$.
\end{remark}

We have the following straightforward modification of \cite[Lemmas~2.69 and~2.84]{Bahouri}, see also \cite[Lemma A.3]{GIP}.
\begin{lemma}
\label{lem:SpectralChunks}
Given a sequence of smooth $(f_j)_{j\geq -1}\in C^\infty(\RR^d;X)$ such that $\supp \FF f_j\subseteq 2^{j\s} \B $ for some (fixed) box $\B$, we have for $\gamma>0$ and $f:=\sum_{j\geq -1} f_j$
\begin{align}
\label{eq:SpectralChunks}
\|f\|_{\Bs^\gamma(\RR^d;X)} \lesssim \left\| \left(2^{j\gamma } \|f_j\|_{L^p(\RR^d;X)}\right)_{j\geq -1} \right\|_{\ell^q}\,.
\end{align}
If $\supp \FF f_j\subseteq 2^{j\s} \rA $ for a rectangular annulus $\rA$, then \eqref{eq:SpectralChunks} is even true for all $\gamma\in \RR$.
\end{lemma}

An intuition behind the anisotropic scaling is that $f\in \Bs^\gamma(\RR^d;X)$ has ``regularity $\alpha / \s_i$ in direction $i\in \{1,\ldots,d\}$''\footnote{Although this intuition is helpful to ``guess'' $\s$ in many situations it is actually slightly incorrect, since the parameter $\gamma$ should really be read in the sense of an average.  A more appropriate (but rather useless) intuition for $\Bs^\gamma$ would be that $f\in \Bs^\gamma$ has in average a regularity of $d\cdot \gamma/|\s|$. Compare the regularity of white noise \cite[Lemma~10.2]{RegularityStructures} as an example where the ``directional intuition'' evidently fails.}.
To strenghten this intuition we will find a different characterization of the Besov spaces $\Bs^\gamma$ based on the Taylor remainder for $\gamma>0$
\glsadd{Rgammaxh}
\begin{align}
 R^\gamma_{x;h}f :=f-T^\gamma_{x;h}f \,, \label{eq:TaylorRemainder}
\end{align}
where 
\glsadd{Tgammaxh}
\begin{align*}
T^\gamma_{x;h}f :=\sum_{k\in \NN^d_{<\gamma}} \frac{1}{k!} \partial^k f(x)\,h^k
\end{align*}
for $x,h\in \RR^d,\,\gamma>0,\,\mathbb{N}^d_{<\gamma}:=\{k\in \mathbb{N}^d\,\vert\,|k|_\s<\gamma\}$ and with $f$ having enough derivatives such that these expressions make sense. $R^\gamma_{x;h} f$ can be rewritten by an application of Proposition~11.1 of \cite{RegularityStructures}.
\glsadd{NNdgammal}

\begin{lemma}
\label{lem:Taylor}
Let $X$ be a Banach space. Let further $\gamma\in (0,\infty)\backslash \NN$ and let $f\in C^\infty(\RR^d;X)$. We then have for $x,h\in \RR^d$ 
\begin{align}
\label{eq:Taylor}
&R^\gamma_{x;h} f= \sum_{k\in \mathbb{N}^d_{> \gamma}} R^{\gamma,k}_{x;h} f=\sum_{k\in \mathbb{N}^d_{> \gamma}} \frac{h^k}{(k-e_{\mathfrak{m}(k)})!}   \int_0^1 \dd t \,\partial^k f(x+v_t^k(h))\, (1-t)^{k_{\mathfrak{m}(k)}-1}   \,,
\end{align}
where $\mathfrak{m}(k)=\min\{j\,  \vert \,k_j\neq 0\}$, $\NN^d_{> \gamma}:=\{k\in \NN^d\,\vert\, |k|_\s > \gamma,\,|k-e_{\mathfrak{m}(k)}|_\s<\gamma\}$ and
\begin{align*}
v_t^k(h)=(h_1,\ldots,h_{\mathfrak{m}(k)-1},t\cdot h_{\mathfrak{m}(k)},0,\ldots,0)\,,
\end{align*}
for the canonical basis $(e_1, \dots, e_d)$ of $\RR^d$.
\glsadd{NNdgammag}
\end{lemma}

\begin{remark} 
The set $\NN^d_{>\gamma}$ can be thought of as the ``discrete boundary'' of $\NN^d_{<\gamma}$. Note that this set is finite because it only contains $k$ with $|k|_\s - \s_{\mathfrak{m}(k)} < \gamma$.
\end{remark}

The announced characterization of anisotropic Besov spaces is given by the following lemma, which is a modification of \cite[Theorem 2.36]{Bahouri}.

\begin{lemma}
\label{lem:Hoelder}
For $\gamma\in (0,\infty) \backslash\NN$ and a Banach space $X$ an equivalent norm for $\Bs^\gamma(\RR^d;X)$ is given by the anisotropic \emph{Hölder norm}
\begin{align}
\label{eq:HoelderNorm}
\sup_{l\in \NN^d_{<\gamma}}\|\partial^l f\|_{C_b(\RR^d;X)}+\sup_{l\in \NN^d_{<\gamma}}\sup_{x,y\in \RR^d,\,0<\|x-y\|_\s\leq 1} \frac{\|\partial^l f(y)-T^{\gamma-|l|_\s}_{x;y-x}\partial^l f\|_X}{\|y-x\|^{\gamma-|l|_\s}_\s}\,.
\end{align}
where we recall that $\|g\|_{C_b(\RR^d;X)}:=\sup_{x\in \RR^d} \|g(x)\|_{X}$.
\end{lemma}

\begin{remark}
Note that the norm in \eqref{eq:HoelderNorm} is equivalent to
\begin{align}
\label{eq:HoelderNormGlobal}
\sup_{l\in \NN^d_{<\gamma}}\|\partial^l f\|_{C_b(\RR^d;X)}+\sup_{l\in \NN^d_{<\gamma}}\sup_{x,y\in \RR^d,x\neq y} \frac{\|\partial^l f(y)-T^{\gamma-|l|_\s}_{x;y-x}\partial^l f\|_X}{\|y-x\|^{\gamma-|l|_\s}_\s}\,,
\end{align}
since for $\|x-y\|_s > 1$ the second term of \eqref{eq:HoelderNormGlobal} can be bounded via the first term of \eqref{eq:HoelderNorm}.  
\end{remark}

\begin{remark}
	The restriction $\gamma \notin \NN$ is not a shortcoming of our proof: The equivalence of the norms really fails for integers $\gamma$, and $\|\cdot \|_{\Bs^\gamma(\RR^d;X)}$ is instead equivalent to a ``Zygmund type'' norm, see~\cite[Theorem~2.37]{Bahouri} for a result in that direction.
\end{remark}

\begin{proof}
Assume that $f\in \Bs^\gamma(\RR^d;X)$ as defined in Definition \ref{def:Besov} above and further, without loss of generality, that $\|f\|_{\Bs^\gamma(\RR^d;X)}\leq 1$. If we write $\overline{\varDelta}_j f:=\sum_{i:\,|i-j|\leq 1} \varDelta_i f$, we have by spectral support properties $\varDelta_j f=\varDelta_j \overline{\varDelta}_j f=\Psi^j \ast \overline{\varDelta}_j f$. Indeed, by construction of $(\varphi_j)_{j\geq -1}$ we have $\varphi_{j}\cdot \sum_{i:\,|j-i|\leq 1} \varphi_{i}=\varphi_j\cdot 1=\varphi_j$ so that by our definition of $\varDelta_j$
\begin{align*}
\FF (\varDelta_j \overline{\varDelta}_j f)=\varphi_{j}\sum_{i:\,|j-i|\leq 1} \varphi_{i} \,\FF f=\varphi_j\, \FF f=\FF(\varDelta_j f)\,,
\end{align*}
from which the claimed identity follows. With Lemma \ref{lem:PsiScal} and Young's convolution inequality we obtain
\begin{align}
\|\partial^l \varDelta_j f\|_{C_b(\RR^d;X)} &\oset{$(\ast)$}{=} \|\partial^l \varDelta_j f\|_{L^\infty(\RR^d;X)} =\| \partial^l \Psi^j\ast \overline{\varDelta}_j f\|_{L^\infty(\RR^d;X)}\nonumber \\
 &\lesssim \|\partial^l \Psi^j\|_{L^1(\RR^d)}\,\|\overline{\varDelta}_j f\|_{L^\infty(\RR^d;X)} 
\oset{Lem. \ref{lem:PsiScal}}{\lesssim} 2^{j|l|_\s} 2^{-j\gamma}=2^{-j(\gamma-|l|_\s)}\,,
\label{eq:DerivativeBlockEstimate}
\end{align}
where we used in $(\ast)$ that $\varDelta_j f$ is smooth and in particular continuous. Decomposing $f=\sum_{j\geq -1} \varDelta_j f$ this implies that the first term of \eqref{eq:HoelderNorm} is bounded
\begin{align*}
\|\partial^l f\|_{C_b(\RR^d;X)}\leq  \sum_{j\geq -1} \|\partial^l \varDelta_j f\|_{C_b(\RR^d;X)}   \lesssim \sum_{j>-1} 2^{-j(\gamma-|l|_\s)}\lesssim 1\,.
\end{align*}
 
To estimate the second term of \eqref{eq:HoelderNorm} we consider for $j\geq -1$ and $x,\,y\in \RR^d$ with $0<\|x-y\|_\s\leq 1$
\begin{align*}
& (\partial^l \varDelta_j f)_y-\sum_{k\in \NN^d_{<\gamma-|l|_\s}}\frac{(\partial^{k+l} \varDelta_j f)_x}{k!} (y-x)^k \\
&\hspace{40pt} =\intr \dd u\, \Big(\partial^l \kk_{y-u}-\sum_{k\in \NN^d_{<\gamma-|l|_\s}} \frac{\partial^{k+l} \kk_{x-u}}{k!} (y-x)^k \Big) \, (\overline{\Delta}_j f)_u\,,
\end{align*}
where we used once more that $\varDelta_j f=\varDelta_j\overline{\varDelta}_j f=\Psi^j\ast \overline{\varDelta}_j f$ for $\overline{\varDelta}_j f =\sum_{i:\,|i-j|\leq 1} \varDelta_{i} f$ as above. Formula \eqref{eq:Taylor} for the Taylor remainder then gives
\begin{align*}
 = \sum_{k\in \NN^d_{> \gamma-|l|_\s}} \frac{(x-y)^k}{(k-e_{\mathfrak{m}(k)})!}  \int_0^1 \dd t\,\intr \dd u  \,\partial^{k+l} \Psi^j_{x-u+v_t^k(y-x)} (1-t)^{k_{\mathfrak{m}(k)}-1} \,(\overline{\varDelta}_j f)_u\,.
\end{align*}
With Young's inequality we thus obtain the bound
\begin{align}
&\Big\| (\partial^l\varDelta_j f)_y - \sum_{k\in \NN^d_{<\gamma-|l|_\s}}\frac{(\partial^{k+l}  \varDelta_j f)_x}{k!} (y-x)^k \Big\|_X     \nonumber \\
&\hspace{20pt}\overset{\|y-x\|_\s \le 1}{\lesssim} \sum_{k\in \NN^d_{> \gamma-|l|_\s}} \|y-x\|^{|k|_\s}_\s \int_0^1 \dd t\,\|(\partial^{k+l} \Psi^j\ast\overline{\varDelta}_j f)_{x+v^k_t(y-x)}\|_X  \nonumber \\
&\hspace{20pt}\lesssim \sum_{k\in \NN^d_{> \gamma-|l|_\s}} \|y-x\|^{|k|_\s}_\s 
\int_0^1 \dd t\,\| \overline{\varDelta}_j f\|_{L^\infty(\RR^d;X)}\,\|\partial^{k+l} \Psi^j\|_{L^1(\RR^d)} \nonumber \\
&\hspace{20pt}\lesssim \sum_{k\in \NN^d_{> \gamma-|l|_\s}} \|x-y\|_\s^{|k|_\s} 2^{j(|k|_\s+|l|_\s-\gamma)} \,,
\label{eq:Hoelder1}
\end{align}
where we applied $\|\overline{\varDelta}_j f\|_{L^\infty(\RR^d;X)}\lesssim 2^{-j\gamma} $ and $\|\partial^{k+l} \Psi^j\|_{L^1(\RR^d)}\lesssim 2^{j(|k|_s+|l|_\s)}$ (by Lemma \ref{lem:PsiScal}) in the last step.

On the other hand, by \eqref{eq:DerivativeBlockEstimate}, we have the easy estimate
\begin{align}
\label{eq:Hoelder2} \nonumber
&\Big\|(\partial^l \varDelta_j f)_y-\sum_{k\in \NN^d_{<\gamma-|l|_\s}}\frac{(\partial^{k+l}  \varDelta_j f)_x}{k!} (y-x)^k \Big\|_X \\
&\hspace{50pt} \lesssim 2^{-j(\gamma-|l|_\s)} +\sum_{k\in \NN^d_{<\gamma-|l|_\s}}2^{-j(\gamma-|l|_\s-|k|_\s)} \|y-x\|_\s^{|k|_\s} \,.
\end{align}
Next, we decompose the Taylor expansion in a ``low-frequency'' and a ``high-frequency'' term. That is, choose $j'\geq -1$ such that $2^{-j'-1}< \|y-x\|_\s \leq 2^{-j'}$ and split
\begin{align*}
\partial^l f_y-\sum_{k\in \NN^d_{<\gamma-|l|_\s}} \frac{(\partial_{k+l} f)_x}{k!} (y-x)^k &=\sum_{j:\,j\leq j'}  \big[(\partial^l \varDelta_j f)_y-\sum_{k\in \NN^d_{<\gamma-|l|_\s}}\frac{(\partial^{k+l}  \varDelta_j f)_x}{k!} (y-x)^k \big]\\
&\quad +\sum_{j:\,j>j'}\big[ (\partial^l \varDelta_j f)_y-\sum_{k\in \NN^d_{<\gamma-|l|_\s}}\frac{(\partial^{k+l}  \varDelta_j f)_x}{k!} (y-x)^k\big]\,.
\end{align*}
Applying now \eqref{eq:Hoelder1} to the first and \eqref{eq:Hoelder2} to the second term yields the first direction of the equivalence of the norms:
\begin{align*}
&\Big\|\partial^l f_y-\sum_{k\in \NN^d_{<\gamma-|l|_\s}} \frac{(\partial^{k+l} f)_x}{k!} (y-x)^k \Big\|_X  \\
 &\lesssim \sum_{j:\,j\leq j'} \sum_{k\in \NN^d_{> \gamma-|l|_\s}} \|x-y\|_\s^{|k|_\s} 2^{j(|k|_\s+|l|_\s-\gamma)}   + \sum_{j:\,j>j'} \Big( 2^{-j(\gamma-|l|_\s)} +\sum_{k\in \NN^d_{<\gamma-|l|_\s}}2^{-j(\gamma-|l|_\s-|k|_\s)} \|y-x\|_\s^{|k|_\s} \Big) \\
&\lesssim \sum_{k\in \NN^d_{> \gamma-|l|_\s}} \|x-y\|_\s^{|k|_\s} 2^{j'(|k|_\s+|l|_\s-\gamma)}
+  \Big( 2^{-j'(\gamma-|l|_\s)} +\sum_{k\in \NN^d_{<\gamma-|l|_\s}}2^{-j'(\gamma-|l|_\s-|k|_\s)} \|y-x\|_\s^{|k|_\s} \Big) \\ 
& \lesssim 
\|y-x\|_\s^{\gamma-|l|_\s}\,.
\end{align*}
For the opposite direction suppose without loss of generality that \eqref{eq:HoelderNormGlobal} is bounded by 1. For $j> -1$ we then have
\begin{align*}
\|(\varDelta_j f)_x\|_X &=\big\| \intr \dd u \Psi^j_{x-u} f_u\big\|_X\oset{Lem. \ref{lem:PsiPol}}{=}\big\|\intr \dd u\, \Psi^j_{x-u} \big(f_u-\sum_{k\in \NN^d_{<\gamma}}\frac{(\partial^k f)_x}{k!} (u-x)^k\big)\big\|_X \\
&\leq \intr \dd u\, |\Psi^j_{x-u}|\cdot \|f_u-T^\gamma_{x;u-x}f\|_X\oset{Lem. \ref{lem:PsiPolScal}}{\lesssim} 2^{-j\gamma}\,,
\end{align*}
where we used Lemma \ref{lem:PsiPol} to introduce a term $\sum_{k\in \NN^d_{<\gamma}}\frac{(\partial^k f)_x}{k!}\intr \dd u\, \Psi^j_{x-u}  (u-x)^k=0$. To bound $\Delta_{-1} f$ we apply \eqref{eq:LpLp} to get $\|\varDelta_{-1} f\|_{L^\infty(\RR^d;X)}\lesssim \|f\|_{L^\infty(\RR^d;X)}=\|f\|_{C_b(\RR^d;X)}\leq 1$, which shows that $\| f\|_{\Bs^\gamma(\RR^d;X)} \lesssim 1$.
\end{proof}

\begin{lemma}
\label{lem:ActionDerivative}
Let $X$ be a Banach space, let $\gamma\in \RR$ and $k\in \NN^d$. We have for $f\in \Bs^\gamma(\RR^d;X)$
\begin{align*}
\|\partial^k f\|_{\Bs^{\gamma-|k|_\s}(\RR^d;X)} \lesssim \| f\|_{\Bs^{\gamma}(\RR^d;X)}.
\end{align*}
\end{lemma}
\begin{proof}
As in the beginning of the proof of Lemma \ref{lem:Hoelder} we use the estimate \eqref{eq:DerivativeBlockEstimate}:
\begin{align*}
\|\varDelta_j \partial^k f\|_{L^p(\RR^d;X)}=\|\partial^k\varDelta_j  f\|_{L^p(\RR^d;X)}\lesssim 2^{j|k|_\s}\,\|\overline{\varDelta}_j f\|_{L^p(\RR^d;X)}\,.
\end{align*}
with $\overline{\varDelta}_j f=\sum_{i:\,|i-j|\leq 1} \varDelta_i f$, which implies the claim. 
\end{proof}

\subsection{Basics of regularity structures}
\label{subsec:RegularityStructures}
In this subsection we recall the basic definition of a regularity structure as a graded vector space equipped with a linear group. We then give a recap on models and on modelled distributions, a concept that is central to this article.

\label{sec:RegularityStructures}
\begin{definition}\cite[Definition 2.1]{RegularityStructures}
\label{def:RegularityStructure} 
A regularity structure is a triple $\mathscr{T}=(A,\mathcal{T},G)$ consisting of:
\begin{itemize}
\item A locally finite \emph{index set} $A\subseteq \RR$, bounded from below, such that $0\in A$.
\item A \emph{model space} $\mathcal{T}=\bigoplus_{\alpha\in A} \mathcal{T}_\alpha$, where each $\mathcal{T}_\alpha$ is a Banach space equipped with a norm $\|\cdot\|_{\tT_\alpha}$. The space $\mathcal{T}_0$ is spanned by a unit vector which we call $\mathbf{1}$.
\item A \emph{structure group} $G$ of linear operators acting on $\tT$, such that for all $\Gamma\in G,\,\alpha\in A,\,\tau\in \tT_\alpha$
\begin{align}
\label{eq:GroupActionHomogeneousSpace}
 \Gamma \tau -\tau\in \bigoplus_{\beta < \alpha} \tT_\beta\,.
 \end{align} 
\end{itemize}
 The elements of $A$ are called \emph{homogeneities}.
 \glsadd{T}
\end{definition}
Given $\tau \in \tT$ and $\alpha\in A$ we write $\tau^\alpha$ for the projection of $\tau$ on $\tT_\alpha$. If $\dim\,\tT_\alpha<\infty$ and we have a basis $\{ e_i\}$ for $\tT_\alpha$, we also write $\tau^{e_j}$ for the coefficient of $\tau^\alpha$ with respect to $e_j$. For example if $\tau\in \tT$ with $\tau -c \cdot \uu \in \bigoplus_{\alpha \in A,\,\alpha \neq 0} \tT_\alpha$ we have $\tau^\uu=c$.
\glsadd{taualpha}
\glsadd{tauuu}
\label{tauuu}
Note that by this definition $\tau^\alpha\in \tT_\alpha$ is a vector in a Banach space, while $\tau^{e_i}$ is a complex number. We also write 
\begin{align*}
\tau^{<\gamma}:=\sum_{\alpha \in A:\,\alpha<\gamma} \tau^\alpha\,,
\end{align*}
and similarly for $\tau^{>\gamma}$. For $\Gamma\in G$  we use the abbreviation 
\glsadd{Gammaalphatau}
\begin{align}
\label{eq:Gammaalphatau}
\Gamma^{\alpha}\tau:=\left( \Gamma \tau\right)^{\alpha}\,.
\end{align}
The same remark applies for the ``basis notation'' above, for instance $\Gamma^{\uu} \tau:=(\Gamma \tau)^\uu$. 
We will also need the space
\begin{align*}
\tT_\gamma^-:=\bigoplus_{\alpha\in A:\,\alpha<\gamma} \tT_\alpha
\end{align*}
for $\gamma\in \RR$, so that for example $\tau^{<\gamma} \in \tT_\gamma^-$ for $\tau \in \tT$.

Let us now introduce the notion of a model.

\begin{definition}{\cite[Definition 2.17]{RegularityStructures}}
\label{def:Model}
Let $\T=(A,\tT,G)$ be a regularity structure and let $\s\in \NNx^d$ be a scaling vector. 
A model on $\T$ with scaling $\s$ is a family of linear maps $\Gamma_{xy}\in G,\,\Pi_x:\,\tT\rightarrow \mathcal{S}'(\RR^d)$ for $x,\,y\in \RR^d$ that satisfy for $x,y,z\in \RR^d$
\begin{align}
\label{eq:ModelChensRelation}
\Gamma_{xx}=\mathrm{Id}_{\tT},\,\Gamma_{xy}\Gamma_{yz}=\Gamma_{xz},\,\Pi_{x}=\Pi_{y}\Gamma_{yx}\,,
\end{align}
and further for $\alpha,\beta\in A,\,\tau \in \tT_\alpha$ and $\beta<\alpha$ 
\begin{align}
&\|\Gamma^{\beta}_{yx}\tau\|_{\tT_\beta} \lesssim \, \|\tau\|_{\tT_\alpha}\cdot \|x-y\|_\s^{\alpha-\beta}, \label{eq:BoundGamma} \\
&|\Pi_x \tau (\varphi^\lambda_{\plc-x})| \lesssim \,\|\tau\|_{\tT_\alpha}\cdot  \lambda^{\alpha}\,, \label{eq:BoundPi}
\end{align}
with $\varphi^\lambda_{\plc-x}=\lambda^{-|\s|}\varphi(\lambda^{-\s}(\cdot-x))$, uniformly over all $\lambda\in (0,1]$ and $\varphi\in C^\infty_c(\RR^d),\,\supp\,\varphi\subseteq B_\s(0,1),\,\|\varphi\|_{C^{r}}\leq 1$, with $r\in \NN$ being the smallest number strictly bigger than $-\min A$. As in \eqref{eq:Gammaalphatau} we wrote $\Gamma^\beta_{yx}\tau:=(\Gamma_{yx}\tau)^\beta$ for the projection of $\Gamma_{yx} \tau$ onto $\tT_\beta$.

We further introduce:
\begin{align*}
\|\Pi\|_\gamma &:=\sup_{\varphi}\,\, \sup_{\alpha \in A,\,\alpha<\gamma\,, \tau \in \tT_\alpha,\,\|\tau\|_{\tT_\alpha}\leq 1 } \,\, \sup_{\lambda\in (0,1]} \lambda^{-\alpha}|\Pi_x \tau (\varphi^\lambda_{\plc-x})|\,,\\
\|\Gamma\|_\gamma &:=\sup_{x,\,y\in \RR^d,\,x\neq y} \,\, \sup_{\alpha,\beta\in A,\,\beta<\alpha<\gamma,\,\tau \in \tT_\alpha,\,\|\tau\|_{\tT_\alpha}\leq 1} \|\Gamma^{\beta}_{yx}\tau \|_{\tT_\beta} \|x-y\|_\s^{\beta-\alpha}\,, 
\end{align*}
where $\sup_\varphi$ runs over the class of $\varphi$ described above.

We sometimes write $\Gamma_{y,x}$  instead of $\Gamma_{yx}$ to separate the arguments more clearly. 
\end{definition}
\begin{remark}
The functions $\Pi_x:\tT\rightarrow \Sw'(\RR^d)$ do not play an important role in this article and are just mentioned for the sake of completeness. However, compare \cite{GIP} and \cite{Dissertation} for similar concepts to the ones presented here  where $\Pi_x$ becomes important.
\glsadd{PiGamma}
\end{remark}
Note that we require global bounds on the model $(\Pi,\Gamma)$ in Definition \ref{def:Model}, which is different from \cite{RegularityStructures} where the corresponding bounds only need to hold locally uniformly on compact sets.
The main reason for requiring global estimates is that we work with an approach based on Fourier analysis, for which it seems unavoidable to work with bounds on the full space. Compare also \cite{ReconstructionBesov} for another work with these assumptions. Global bounds are given immediately in the study of a SPDE if the considered equation is driven by a periodic space white noise. In the case of a spatially periodic space-time white noise we can replace the noise with one that  is also periodic in time,  with a period that is bigger than the time horizon of the equation. If one wants to consider problems with non-periodic noise, one would have to introduce weights in the analytic bounds of Definition~\ref{def:Model}, similarly as in \cite{HairerLabbeR3, WeberMourrat, DiscreteParacontrolled}. We will avoid doing so for the sake of simplicity. 

%

A simple example of a regularity structure equipped with a model is the \emph{polynomial regularity structure} $\overline{\T}=(\overline{A},\overline{\tT},\overline{G})$, where
\glsadd{overlineT}
\begin{align}
\label{eq:PolynomialStructure}
\overline{\tT}:=\Span{X^k \,\big\vert\,k\in \NN^d}\,,
\end{align}
where $\Span{\ldots}$ denotes the vector space generated by the set in the braces and where we identify $X^0=\mathbf{1}$. We assign the homogeneities $|X^k|=|k|_\s$ to the symbols $X^k$ and define $\overline{A}:=\NN$, so that
\begin{align*}
\overline{\tT}=\bigoplus_{\alpha \in \overline{A}} \overline{\tT}_\alpha:=\bigoplus_{\alpha \in \overline{A}} \Span{ X^k\, \big\vert \, k\in \NN^d, \, |k|_\s=\alpha  }\,.
\end{align*}
As each $\overline{\tT}_\alpha=\Span{X^k\,\vert \,k\in \NN^d,\,|k|_\s=\alpha}$ is finite dimensional any choice of norm on $\overline{\tT}_\alpha$ will lead to the same topology; we take
\begin{align*}
\Big\|\sum_{k\in \NN^d:\,|k|_\s=\alpha} a_k\,X^k \Big\|_{\overline{\tT}_\alpha} :=\sum_{k\in \NN^d:\,|k|_\s=\alpha} |a_k|\,.
\end{align*}
Consider then the group $\overline{G}=\{\overline{\Gamma}_{h}\,\vert\,h\in \RR^d\}$ with group law $\overline{\Gamma}_h\overline{\Gamma}_{h'}:=\overline{\Gamma}_{h+h'}$ for $h,\,h'\in \RR^d$ (so that $\overline{G}$ is essentially $(\RR^d,+)$). Fix the action of $\overline{G}$ on $\overline{\tT}$ by requiring $\overline{\Gamma}_h X^k:=(X+h\mathbf{1})^k$ (with the obvious interpretation of the multiplication on the right hand side). 

We can realize a model on $\overline{\T}=(\overline{A},\overline{\tT},\overline{G})$ via 
\begin{align}
\label{eq:PolynomialModel}
\overline{\Pi}_x X^k(y)=(y-x)^k,\qquad \overline{\Gamma}_{yx}:=\overline{\Gamma}_{y-x}\,,
\end{align}
for $x,y\in \RR^d$ and $k\in \NN^d$. The polynomial regularity structure $\overline{\T}$ is the example one should have in mind when it comes to comparison of regularity structures with results from ``more classical'' analysis. From this perspective the spaces $\DD^\gamma$ which we are going to define now (taken from \cite[Definition 3.1]{RegularityStructures}) are a generalization of classical Hölder spaces, compare Lemma \ref{lem:HoelderAsModelled} below. 

\begin{definition}
\label{def:Dgamma}
Let $\T=(A,\tT,G)$ be a regularity structure, equipped with a model $(\Pi,\Gamma)$ with scaling $\s\in \NNx^d$. Given $\gamma\in \RR$ we say that a mapping $F:\,\RR^d \rightarrow \tT_{\gamma}^-$ belongs to $\DD^\gamma(\RR^d;\tT)=\DD^\gamma(\RR^d;\tT,\Gamma)$ if 
\begin{align}
\label{eq:NormDgamma}
\|F\|_{\DD^\gamma(\RR^d;\tT)}:=\sup_{\alpha \in A,\, x\in \RR^d} \|F^\alpha_x\|_{\tT_\alpha}+\sup_{\alpha \in A,\,x,y\in \RR^d,\,x\neq y} \frac{\|F^\alpha_y-\Gamma^\alpha_{yx}F_x \|_{\tT_\alpha}}{\|y-x\|_\s^{\gamma-\alpha}}<\infty\,.
\end{align}
We call the objects in $\DD^\gamma(\RR^d;\tT)$ \emph{modelled distributions}. 
\glsadd{DDgamma}
\end{definition}

\begin{remark}
\label{rem:LongDistancexy}
Due to the bound \eqref{eq:BoundGamma} it is enough to consider in the second term in \eqref{eq:NormDgamma} only pairs $x,y\in \RR^d$ with $\|x-y\|_\s\leq 1$, as long as the first supremum in \eqref{eq:NormDgamma} is finite. 
\end{remark}

\begin{remark}
\label{rem:DgammaGlobal}
Note that all $F\in \DD^\gamma(\RR^d;\tT)$ satisfy \emph{global} bounds, which is different from \cite{RegularityStructures} where the notation $F\in \DD^\gamma(\RR^d;\tT)$ only indicates local bounds. In a framework that is largely based on Fourier analysis such as the paracontrolled approach it seems natural to assume global bounds. Local spaces could then be defined afterwards with the help of extension operators, see e.g. Section~5.3.3 of~\cite{Dissertation} for a version of the Whitney extension theorem in regularity structures.
\end{remark}

The definition of a modelled distribution $F\in \DD^\gamma(\RR^d;\tT)$ implies the continuity of every component  $F^\alpha\in C_b(\RR^d;\tT_\alpha)$ with $\alpha \in A$ and the bound $\sup_{\alpha\in A,\,x\in \RR^d} \|F^\alpha(x)\|_{\tT_\alpha}<\infty$. We will denote functions $F: \RR^d\rightarrow \tT$ that satisfy these two properties by $C_b(\RR^d;\tT)$ and set
\begin{align*}
\|F\|_{C_b(\RR^d;\tT)}:=\sup_{\alpha\in A,\,x\in \RR^d} \|F^\alpha_x\|_{\tT_\alpha}\,,
\end{align*}
so that in particular $\DD^\gamma(\RR^d;\tT)\subseteq C_b(\RR^d;\tT)$. It turns out that on the polynomial regularity structure $\overline{\T}$ it is rather easy to describe $\mathscr{D}^\gamma$ in terms of the Besov spaces from Definition \ref{def:Besov}.

\begin{lemma}
\label{lem:HoelderAsModelled}
Let $\overline{\T}=(\overline{A},\overline{\tT},\overline{G})$ be the polynomial regularity structure with model $(\overline{\Pi},\overline{\Gamma})$ introduced on page \pageref{eq:PolynomialStructure}, for some scaling vector $\s\in \NNx^d$. Define for $\gamma\in \RR_+\backslash \NN$ the lifted Besov space
\begin{align*}
 \PP^{\gamma}(\RR^d;\overline{\tT}) :=\{  F_f\, \vert\, f\in \Bs^\gamma(\RR^d) \}\,,
\end{align*} 
where $F_f:=\sum_{k\in \NN^d_{<\gamma}}\frac{1}{k!}\,\partial^k f\cdot X^k\in C_b(\RR^d;\overline{\tT})$ denotes the lift of $f\in \Bs^\gamma(\RR^d)=\Bs^\gamma(\RR^d;\CC)$ to the polynomial structure. Equip further $\PP^\gamma(\RR^d;\overline{\tT})$ with the norm
\begin{align*}
\|F\|_{\PP^\gamma(\RR^d;\overline{\tT})}:=\sup_{\alpha\in \overline{A}} \|F^{\alpha}\|_{\Bs^{\gamma-\alpha}(\RR^d;\tT_\alpha)}\,,
\end{align*}
where $F^\alpha$ denotes the projection of $F\in \PP^\gamma(\RR^d;\overline{\tT})$ on $\tT_\alpha$. We then have
\begin{align*}
\DD^\gamma(\RR^d;\overline{\tT})=\PP^\gamma(\RR^d;\overline{\tT})
\end{align*}
with equivalent norms.
\end{lemma}

\begin{proof}
Let $F\in \PP^\gamma(\RR^d;\overline{\tT})$. By definition there is a (unique) $f\in \Bs^\gamma(\RR^d)$ such that $F=F_f$. In particular we have for $\alpha\in \overline{A}$
\begin{align}
\label{eq:HoelderAsModelled1}
F^\alpha=F_f^\alpha=\sum_{k\in \NN^d:\,|k|=\alpha}\frac{1}{k!} \partial^k f\cdot X^k
\end{align}
and the norm $\|F\|_{\PP^\gamma(\RR^d;\overline{\tT})}$ is thus equivalent to
\begin{align*}
\|F\|_{\PP^\gamma(\RR^d;\overline{\tT})} \oset{\eqref{eq:HoelderAsModelled1}}{\approx} \sup_{k\in \NN^d_{<\gamma}} \|\partial^k f\|_{\Bs^{\gamma-|k|_\s}(\RR^d)} \oset{Lem. \ref{lem:ActionDerivative}}{\approx} \|f\|_{\Bs^\gamma(\RR^d)}\,.
\end{align*}
With Lemma \ref{lem:Hoelder} we thus get the equivalence 
\begin{align}
\|F\|_{\PP^\gamma(\RR^d;\overline{\tT})} \approx \sup_{l\in \NN^d_{<\gamma}}\|\partial^l f\|_{C_b(\RR^d)}+\sup_{l\in \NN^d_{<\gamma}}\sup_{x,y\in \RR^d,\,0<\|x-y\|_\s\leq 1} \frac{|\partial^l f(y)-T^{\gamma-|l|_\s}_{x;y-x}\partial^l f|}{\|y-x\|^{\gamma-|l|_\s}_\s}\,.
\label{eq:HoelderAsModelled2}
\end{align}
Since $T^{\gamma-|l|_\s}_{x;y-x}\partial^l f=\overline{\Gamma}^{X^l}_{yx} (F_f)_x=\overline{\Gamma}^{X^l}_{yx} F_x$ and $\overline{\Gamma}^\alpha_{yx} F_x = \sum_{l\in \NN^d:\,|l|_\s=\alpha} \overline{\Gamma}^{X^l}_{yx} F_x \cdot X^l$ for $\alpha\in \overline{A}$ we see that the right hand side of \eqref{eq:HoelderAsModelled2} is equivalent to $\|F\|_{\DD^\gamma(\RR^d)}$, which proves that $\PP^\gamma(\RR^d)\subseteq \DD^\gamma(\RR^d)$. The inverse direction follows in the same manner once we show that every $F\in \DD^\gamma(\RR^d)$ is of the form $F=F_f$ for some $f\in \Bs^\gamma(\RR^d)$, which can be checked inductively.
\end{proof}

\section{Paraproducts on regularity structures}

\label{sec:paramodelled}
Taking $X=\CC$ in Lemma \ref{lem:Hoelder} provides two distinct descriptions of the (anisotropic) Besov space $\Bs^\gamma(\RR^d)=\Bs^\gamma(\RR^d;\CC)$ with scaling vector $\s\in \NNx^d$ and regularity $\gamma\in \RR_+\backslash \NN$, given by a Littlewood-Paley characterization in Definition \ref{def:Besov} and a Hölder norm in Lemma \ref{lem:Hoelder}. In Lemma \ref{lem:HoelderAsModelled} these two characterizations were formulated in the framework of the polynomial regularity structure $\mathscr{T}=(\overline{A},\overline{\tT},\overline{G})$ with model $(\overline{\Pi},\overline{\Gamma})$, introduced on page \pageref{eq:PolynomialStructure}. Recall that the statement $F\in \DD^\gamma(\RR^d;\overline{\tT})$ is just saying that 
\begin{align}
\label{eq:InterweavingDgamma}
\|F^\alpha_y-\overline{\Gamma}^\alpha_{yx} F_x\|_{\overline{\tT}_\alpha}\lesssim \|y-x\|_\s^{\gamma-\alpha}
\end{align}
for $\alpha\in \overline{A}$, which can be seen as a Hölder-like characerization similar to the one in Lemma~\ref{lem:Hoelder}. On the other hand we introduced in Lemma \ref{lem:HoelderAsModelled} a ``Besov space'' $\PP^\gamma(\RR^d;\overline{\tT})$ as the space of functions $F:\RR^d\rightarrow \overline{\tT}$ for which it holds  
\begin{align}
\label{eq:InterweavingPgamma}
\|\varDelta_j F^\alpha\|_{L^\infty(\RR^d;\overline{\tT}_\alpha)}\lesssim 2^{-j(\gamma-\alpha)}\,.
\end{align}
and further
\begin{align}
\label{eq:PolynomialStructureCondition}
F^\alpha=\sum_{|k|_\s=\alpha} \frac{\partial^k F^\uu}{k!}\cdot X^k\,,
\end{align}
which, using the model $(\overline{\Pi},\overline{\Gamma})$, can also be written as
\begin{align}
\label{eq:InterweavingStructureCondition}
\partial^k (F^\alpha- \overline{\Gamma}_{\cdot x}^{\alpha} F_x)_x=0 \mbox{ for $k\in \NN^d_{<\gamma-\alpha}$}\,,
\end{align}
where $\partial^k (\overline{\Gamma}_{\cdot x}^{\alpha} F_x)_x$ should be read as the derivative of the map 
$y\mapsto \overline{\Gamma}_{y x}^{\alpha} F_x$ evaluated in the point $x$. Note that the relation \eqref{eq:PolynomialStructureCondition} (or equivalently \eqref{eq:InterweavingStructureCondition}) can be deduced from \eqref{eq:InterweavingDgamma}, but it has to be required explicitly in the definition of $\PP^\gamma(\RR^d;\overline{\tT})$ because condition \eqref{eq:InterweavingPgamma} lacks any description of the interplay between between different $F^\alpha,\,F^{\alpha'}$ with $\alpha\neq \alpha'$. However, requiring \eqref{eq:InterweavingStructureCondition} (or \eqref{eq:PolynomialStructureCondition}) and~\eqref{eq:InterweavingPgamma} we could deduce in Lemma \ref{lem:HoelderAsModelled} that $\DD^\gamma(\RR^d;\overline{\tT})=\PP^\gamma(\RR^d;\overline{\tT})$, even without resorting to the explicit construction of $F_f$ from $f$.
 
 Our aim in this article is to find a ``Besov space'' $\PP^\gamma(\RR^d;\tT)$ on a general regularity structure $(A,\T,G)$ with model $(\Pi,\Gamma)$ which describes the space of modelled distributions $\DD^\gamma(\RR^d;\tT)$. Namely, we want to find a space $\PP^\gamma(\RR^d;\tT)$, described in terms of Littlewood-Paley blocks, such that
\begin{align}
\label{eq:InterweavingBridge}
\DD^\gamma(\RR^d;\tT)=\PP^\gamma(\RR^d;\tT)\,.
\end{align}
Already in the original paper on paracontrolled distributions~\cite{GIP} the authors introduced a certain paraproduct $P(F,\Pi)$ on the regularity structure $(A,\tT,G)$ with model $(\Pi,\Gamma)$ and they conjectured that it might be possible to describe the space $\DD^\gamma(\RR^d;\tT)$ via such objects. We here show that this is indeed the case: We introduce a family of paraproducts $P(F,\Gamma^\alpha)$ and define a space $\PP^\gamma(\RR^d;\tT)$ by requiring instead of \eqref{eq:InterweavingPgamma} 
\begin{align}
\label{eq:InterweavingPgamma2}
\|\varDelta_j (F^\alpha-P(F,\Gamma^\alpha))\|_{L^\infty}\lesssim 2^{-j(\gamma-\alpha)}\,,
\end{align}
(which is just saying $F^\alpha-P(F,\Gamma^\alpha)\in \Bs^{\gamma-\alpha}(\RR^d;\tT_\alpha)$) and the \emph{structure condition} \eqref{eq:InterweavingStructureCondition} (with $\overline{\Gamma}$ replaced by $\Gamma$). Since the paraproducts $P(F,\Gamma^\alpha)$, described in Definition \ref{def:Paraproducts} below, vanish for $F$ with components in the polynomial structure $\overline{\tT}$, the bound~\eqref{eq:InterweavingPgamma2} is indeed a generalization of \eqref{eq:InterweavingPgamma}.

\subsubsection*{Paraproducts}

Let us motivate our definitions with a simple example of a singular SPDE, namely the parabolic Anderson model which reads as the following problem on $[0,\infty)\times \RR^2$: 
\begin{align}
\label{eq:PAM}
(\partial_t-\Delta_{\RR^2}) f=f\cdot \xi
\end{align}
with periodic space white noise $\xi\in \Sw'(\RR^2)$ (and a suitable renormalization that we hide for simplicity). The idea in \cite{GIP} is to define first $I\xi$ to be the time-independent solution to $(\partial_t-\Delta_{\RR^2})I \xi =-\Delta_{\RR^2} I\xi =\xi + \tilde{\xi}$ for an infinitely smooth $\tilde \xi$, and to consider instead of $f$ the object 
\begin{align}
\label{eq:fsharp}
f^{\sharp}:=f-f \para  I\xi
\end{align}
with the paraproduct
\begin{align}
\label{eq:PAMParaproduct}
(f\para I \xi)_x=\!\!\sum_{j\geq -1} \sum_{i<j-1} (\varDelta_i f)_x (\varDelta_j I\xi)_x =\!\! \sum_{j>0} \dintr \dd u\dd v\,\KKu \kkv\, f_u \cdot (I\xi)_v\,,
\end{align} 
where the integration domain for each integral should be read as \emph{space-time}, that is $\RR^{1+2}$. Here we cheat a little bit since in \cite{GIP} a modified paraproduct $\mpara$ in space is considered which allows for a cut-off for negative times. It then turns out that $f^{\sharp}$ solves a ``better'' equation than $f$, which allows to derive a priori estimates and to solve the equation. In \eqref{eq:PAMParaproduct} we now take functions $\kk,\,\KK$ that are constructed as in Subsection \ref{subsec:Besov} with an \emph{anisotropic} scaling $\s$, more precisely we take in \eqref{eq:PAMParaproduct} the parabolic scaling $\s=\s_{\mathrm{par}}=(2,1,1)$, which is one more difference with \cite{GIP}.  

In \cite{RegularityStructures} the problem \eqref{eq:PAM} is solved on a regularity structure (again with $\s=\s_{\mathrm{par}}$), and the solution lives in the subspace spanned by the symbols $\{\I(\Xi)\}\cup \{X^k\,:\,k\in \NN^d\}$ and equipped with a model $(\Pi,\Gamma)$ such that 
\begin{align}
\label{eq:PAMModel}
\Pi_x X^k=\Gamma^\uu_{yx}X^k=(y-x)^k,\qquad \Pi_x\I(\Xi)(y)=\Gamma_{yx}^{\uu}\I(\Xi)=I\xi(y)-I\xi(x)\,.
\end{align}
The solution $f$ to \eqref{eq:PAM} is represented by a modelled distribution $F$ of the form 
\begin{align}
\label{eq:PAMModelled}
F=f\,\uu+f \,\I(\Xi) + \sum_{k\in \NN^d:\,|k|_{\s_{\mathrm{par}}=1}} f^{X^k} X^k\,,
\end{align} 
where $f^{X^k}$ are some real valued functions and $f$ is the solution to \eqref{eq:PAM}. Recall from Lemma~\ref{lem:PsiPol} that for $j>0$ the kernel $\Psi^j$ integrates polynomials (and constants) to $0$, and therefore we can rewrite the paraproduct \eqref{eq:PAMParaproduct} in terms of $F$ and the model $(\Pi,\Gamma)$ as 
\begin{align}
\label{eq:PAMFancyParaproduct}
(f\para I\xi)_x =\sum_{j>0} \dintr \dd u\dd v\,\KKu \kkv \Gamma_{vu}^\uu F_u=:P(F,\Gamma^\uu)\,.
\end{align}

This motivates the following definitions.
\begin{definition}
\label{def:Paraproducts}
Let $\T=(A,\tT,G)$ be a regularity structure, let $(\Pi,\Gamma)$ be a model with scaling $\s$ and let $\kk,\,\KK\in \Sw(\RR^d)$ be functions as in \eqref{eq:Psij} (for the same scaling $\s$). We define the following \emph{paraproducts}
\begin{align}
P(F,\Gamma^\alpha)_x &=\sum_{j>0} \dintr \dd u\dd v  \,\KKu \kkv \,\Gamma_{vu}^\alpha F_u \label{eq:ParaproductGamma}
\end{align}
for any $F:\RR^d \rightarrow \tT$ and $\alpha \in A$ for which this is defined. The identity should be read in $\Sw'(\RR^d;\tT_\alpha)$ and is written in formal notation. 
\end{definition}
\begin{remark}
\label{rem:BasisParaproducts}
If $\tT_\alpha$ is finite dimensional for $\alpha\in A$ and we have a basis $\{e_i\}$ for $\tT_\alpha$, we will also write
\begin{align*}
 P(F,\Gamma^{e_i}) :=\sum_{j>0} \dintr \dd u\dd v  \,\KKu \kkv \,\Gamma_{vu}^{e_i} F_u\,,
 \end{align*} 
 where we recall that $\Gamma_{vu}^{e_i} F_u$ denotes the coefficient of $\Gamma_{vu}^{\alpha} F_u \in \tT_\alpha$ in front of $e_i$. In particular we have
 \begin{align*}
 P(F,\Gamma^\alpha)=\sum_{i} P(F,\Gamma^{e_i}) \cdot e_i\,.
 \end{align*}
 For $\alpha=0$ we simply have
 \begin{align*}
 P(F,\Gamma^0)=P(F,\Gamma^\uu)\cdot \uu\,,
 \end{align*}
a notation we already applied in \eqref{eq:PAMFancyParaproduct}.
\end{remark}
\begin{remark}
\label{rem:ParaproductWellDefined}
For measurable and at most polynomially growing $F$ (that is $\|F^\alpha(x)\|_{\tT_\alpha}$ grows at most polynomially in $x\in\RR^d$) the expression \eqref{eq:ParaproductGamma} is well defined. Indeed: Each of the terms in the sum in \eqref{eq:ParaproductGamma} is spectrally supported in a rectangular annulus of the form $2^{j\s}\rA$ (where the rectangular annulus $\rA$ can be chosen independently of $f$ and $j$). This can be easily checked for smooth $(x,y)\mapsto \Gamma_{yx}F_x$, so that the general case follows by approximation. Further by Definition \ref{def:Model} and Lemma \ref{lem:PsiPolScal} one easily sees that each of the terms can be bounded by $2^{-j\kappa}$ for some $\kappa>0$. Lemma \ref{lem:SpectralChunks} then shows that $P(F,\Gamma^\alpha)$ is contained in the Besov space $\Bs^\kappa(\RR^d;\tT_\alpha)$ so that $P(F,\Gamma^\alpha)$ is in fact a (slightly) Hölder continuous function with values in $\tT_\alpha$. 
\glsadd{PFPi}
\end{remark}

\subsubsection*{The space $\PP^\gamma(\RR^d;\tT)$}

Let us come back to our toy example, the parabolic Anderson model. The function $f^{X^k}$ from \eqref{eq:PAMModelled} did not appear in the approach of~\cite{GIP}, which is due to the fact  that according to Lemma~\ref{lem:PsiPol} polynomials are erased in the paraproduct \eqref{eq:PAMFancyParaproduct}. To find a link between the ideas of \cite{GIP} and \cite{RegularityStructures} we therefore need an extra ingredient that forces the $f^{X^k}$ to enter the game. This will be the task of the \emph{structure condition}, which we already motivated for the polynomial framework in \eqref{eq:InterweavingStructureCondition}. 

\begin{definition}
\label{def:StructureCondition}
Let $\T=(A,\tT,G)$ be a regularity structure with a model $(\Pi,\Gamma)$ and scaling $\s\in \NN^d$. Construct the functions $(\Psi^{<N})_{N\geq 0}$, using the scaling vector $\s$, as in \eqref{eq:Psij}. We say that $F:\RR^d\rightarrow \tT$ satisfies the \emph{structure condition below $\gamma\in \RR$} if for all $x\in \RR^d,\,k\in \NN^d$ and $\alpha\in A$ with $\alpha<\gamma$ and $|k|_\s < \gamma-\alpha$ 
and for all large $N \in \NN$ the map $v\mapsto \partial^k\Psi_{x-v}^{< N} \,(F_v^\alpha-\Gamma_{vx}^\alpha F_x)$ is in $L^1(\RR^d;\tT_\alpha)$ and it satisfies
\begin{align}
\lim_{N\rightarrow \infty} \intr \dd v \,\partial^k\Psi_{x-v}^{< N} \,(F_v^\alpha-\Gamma_{vx}^\alpha F_x) = 0\,,
\label{eq:StructureCondition}
\end{align}
where the limit is taken in $\tT_\alpha$.
\glsadd{structurecondition}
\end{definition}

\begin{remark}
\label{rem:StructureConditionSmooth}
Note that we know from Lemma \ref{lem:PsiPol} that $\intr \dd u\,\Psi^{<N}_u=1$ so that we see from the scaling property in Lemma \ref{lem:PsiScal} that $(\Psi^{<N})_{N\geq 0}$ is a (signed) Dirac sequence. From this fact, or alternatively from the decomposition \eqref{eq:LittlewoodPaleyDecomposition}, we conclude that for smooth $F$ and $\Gamma$ condition \eqref{eq:StructureCondition} translates into 
\begin{align}
\label{eq:SmoothStructureCondition}
\partial^k(F^\alpha -\Gamma^\alpha_{\cdot x}F_x)_x=0\,,
\end{align} 	
which is just the identity we announced in \eqref{eq:InterweavingStructureCondition}. The reason why we do not require \eqref{eq:SmoothStructureCondition}	from the start is that typically the smoothness assumption is neither satisfied for the modelled distribution $F$ nor for the map $y\mapsto \Gamma_{yx} F_x$, as one may observe in \eqref{eq:PAMModel}. 
\end{remark}

Let us see how to use the structure condition in order to determine the coefficient $f^{X^k}$ in \eqref{eq:PAMModelled}. Formally using the structure condition in the form \eqref{eq:SmoothStructureCondition}, and ignoring possible smoothness issues for the sake of simplicity, we obtain for $\alpha=0$ and $k\in \NN^d$ with $|k|_{\s_{\mathrm{par}}}=1$ via \eqref{eq:PAMModelled} and \eqref{eq:PAMModel}
\begin{align*}
0\oset{\eqref{eq:SmoothStructureCondition} for $\alpha=0$}{=} (\partial^k f)_x \uu -f_x (\partial^k I\xi)_x \uu-k! f^{X^k}_x \uu\,,
\end{align*}
so that the only possible choice for $f^{X^k}$ is given by
\begin{align*}
f^{X^k}_x=\frac{1}{k!}\big((\partial^k f)_x -f_x \cdot (\partial^k I\xi)_x \big)\,.
\end{align*}
Although this argument should be executed with more care to guarantee that all expressions are well defined\footnote{A rigorous approach would consider \eqref{eq:StructureCondition} instead and show the existence of the limit using the paracontrolled structure of $f$. Compare the proof of \cite[Theorem 6.2.3]{Dissertation}.}, we can already observe that Definition \ref{def:StructureCondition} really fixes $f^{X^k}$ in terms of $f$. Equipped with the structure condition there is now hope that we might find a link between the paracontrolled approach in \cite{GIP} and the description via regularity structures.

As we have already indicated in \eqref{eq:fsharp} the object that is really considered in \cite{GIP} is the difference $f^\sharp$ of the solution $f$ of the considered equation with a paraproduct. In view of Definition \ref{def:Paraproducts} we therefore propose the following definition. 

\begin{definition}
\label{def:Pgamma}
Let $\gamma\in \RR$, let $\T=(A,\tT,G)$ be a regularity structure and let $(\Pi,\Gamma)$ be a model with scaling $\s\in \NN^d$. We say that $F:\,\RR^d\rightarrow \tT^-_\gamma$ with $F\in C_b(\RR^d;\tT)$ is in $\PP^\gamma(\RR^d;\tT)=\PP^\gamma(\RR^d;\tT,\Gamma)$ if for all $\alpha \in A$ with $\alpha < \gamma$
\begin{align}
\label{eq:Pgamma}
F^{\sharp,\alpha}:=F^\alpha -P(F,\Gamma^\alpha)\in \Bs^{\gamma-\alpha}(\RR^d;\tT_\alpha)
\end{align}
and if $F$ satisfies the structure condition \eqref{eq:StructureCondition} below $\gamma$. We define the semi-norm 
\begin{align*}
\|F\|_{\PP^\gamma(\RR^d;\tT)}:=\|F\|_{C_b(\RR^d;\tT)}+\sup_{\alpha\in A} \|F^{\sharp,\alpha}\|_{\Bs^{\gamma-\alpha}(\RR^d;\tT_\alpha)}\,.
\end{align*}
\end{definition}

As we already pointed out above, Definition \ref{def:Paraproducts} is strictly speaking not a generalization of the approach of \cite{GIP}, because we use space-time
paraproduct \eqref{eq:PAMParaproduct} instead of their modified paraproduct $\mpara$. The space-time paraproduct might seem more natural, but it comes with a price: Since the solutions to parabolic equations like \eqref{eq:PAM} are only defined for positive times, we have to extend them to negative times  to fit them into our framework. We thus need extension operations (and spaces that allow for a blow-up around $t=0$) in order to practically apply paraproduct techniques to differential equations in regularity structures. We will not deal with these technical issues here and refer to \cite{Dissertation} for a few results and concepts in that direction. 


\section{Controlling modelled distributions via paraproducts}
\label{sec:MainTheorem}

We now state and prove the main result of this article. We show that the spaces $\DD^\gamma(\RR^d,\tT)$ from Definition \ref{def:Dgamma} and the space $\PP^\gamma(\RR^d;\tT)$ introduced in Definition \ref{def:Pgamma} are identical with equivalent norms. For technical reasons we have to exclude the case that $\gamma\in \RR$ is contained in the locally finite set
\glsadd{AN}
\begin{align}
\label{eq:AN}
\AN:=A+\NN\,.
\end{align}
This is necessary since we want to apply for the spaces $\Bs^{\gamma-\alpha}(\RR^d;\tT_\alpha)$, appearing in the Definition of \ref{def:Pgamma}, the Hölder characterization from Lemma \ref{lem:Hoelder}. We can interpret the following theorem as a generalization of Lemma \ref{lem:Hoelder}, so that the exclusion of $A_\NN$ corresponds to the restriction $\gamma\not\in \NN$ required there. We do not expect that the result holds for $\gamma \in A_\NN$.

\begin{theorem}
\label{thm:BridgeTheorem}
Let $\T=(A,\tT,G)$ be a regularity structure and let $(\Pi,\Gamma)$ be a model with scaling $\s\in \NNx^d$. We then have for any $\gamma\in \RR \backslash \AN$
\begin{align}
\label{eq:Bridge}
\DD^\gamma(\RR^d;\tT)=\PP^\gamma(\RR^d;\tT)
\end{align}
with equivalent norms (where the equivalence constants can both be chosen proportionally to some polynomial in $\|\Gamma\|_\gamma$). 
\end{theorem}

\begin{proof}
\allowdisplaybreaks
We assume without loss of generality that $A$ contains only homogeneities below $\gamma$. We will include polynomials in $\|\Gamma\|_\gamma$ in the implicit constant indicated by ``$\lesssim$'' and we omit the domain ``$\RR^d$'' under integration signs. 
 
For the easy direction $\DD^\gamma(\RR^d;\tT)\subseteq \PP^\gamma(\RR^d;\tT)$ note first that $F\in \DD^\gamma(\RR^d;\tT)$ already implies that the structure condition \eqref{eq:StructureCondition} is satisfied below $\gamma$: Indeed, for $\alpha\in A$ with $\alpha<\gamma$ and $k\in \NN^d_{<\gamma-\alpha}$ we have
\begin{align*}
\left\|\int \dd v\, \,\partial^k \Psi^{<N}_{x-v} (F^\alpha_v-\Gamma^\alpha_{vx}F_x)\right\|_{\tT_\alpha} \overset{\mbox{\tiny Lem. \ref{lem:PsiPolScal}}} \lesssim \|F\|_{\DD^\gamma(\RR^d;\tT)} \cdot 2^{-N(\gamma-\alpha-|k|_\s)} \overset{N\rightarrow \infty}{\rightarrow} 0\,.
\end{align*}
To derive the analytic bound we follow similar ideas as in \cite[Subsection 6.2.]{GIP}: We can rewrite for $x\in \RR^d$ and $\alpha \in A$
\begin{align*}
F^\alpha_x-P(F,\Gamma^\alpha)_x=\sum_{j>0} \int \dd v\,\Big( \kkv  F^\alpha_v -\int \dd u\KKu\,\kkv \Gamma^\alpha_{vu} F_u \Big) + (\Delta_{\leq 0}F^\alpha)_x\,.
\end{align*}
As $\Delta_{\leq 0}F^\alpha=\Psi^{\leq 0} \ast F^\alpha$ is smooth with bounded derivatives we only have to consider the first term on the right hand side. 
We already noted in Remark~\ref{rem:ParaproductWellDefined} that the $j$-th summand is spectrally supported in an annulus $2^{j \s} \rA$, so by Lemma~\ref{lem:SpectralChunks} it suffices to bound it by $  2^{-j(\gamma-\alpha)} \nDl{F}$. We have, using $ \int \KKu 1=1$ (Lemma \ref{lem:PsiPol}),
\begin{align*}
\int \dd v\, \Big(\kkv F^\alpha_v - \int \dd u \KKu \kkv \Gamma^\alpha_{vu} F_u\Big)=\dint \dd u \dd v\, \KKu \kkv (F^\alpha_v-\Gamma^\alpha_{vu} F_u)\,.
\end{align*}
Now, by assumption 
\[
\|F^\alpha_v-\Gamma^\alpha_{vu} F_u\|_{\tT_\alpha}\lesssim \nDl{F} \|u-v\|_\s^{\gamma-\alpha}\lesssim \nDl{F} (\|u-x\|_\s^{\gamma-\alpha}+\|v-x\|_\s^{\gamma-\alpha})\,,
\]
so that we have with Lemma \ref{lem:PsiPolScal} (and Lemma \ref{lem:PsiScal}) the estimate 
\begin{align*}
 \left\|\dint \dd u \dd v\, \KKu \kkv (F^\alpha_v-\Gamma^\alpha_{vu} F_u)\right\|_{\tT_\alpha}\lesssim \nDl{F} 2^{-j(\gamma-\alpha)}\,,
 \end{align*}
which proves $\DD^\gamma(\RR^d;\tT)\subseteq \PP^\gamma(\RR^d;\tT)$. 

Let us now address the delicate direction of the proof, that is $\PP^\gamma(\RR^d;\tT)\subseteq \DD^\gamma(\RR^d;\tT)$. Let $F\in \PP^\gamma(\RR^d;\tT)$ and assume without loss of generality that $\nPl{F}\leq 1$. We will show by induction over decreasing homogeneities in $A$ that for all $x,y\in \RR^d$ and $\alpha\in A$
\begin{align*}
\|F^\alpha_y-\Gamma^\alpha_{yx} F_x\|_{\tT_\alpha}\lesssim    \|y-x\|_\s^{\gamma-\alpha} \,,
\end{align*}
which proves the claim. 
Note that it suffices to take $\|x-y\|_\s\leq 1$, compare Remark \ref{rem:LongDistancexy}.
We start the induction with $\alpha=\max A$ (which exists since we assumed $\max A<\gamma<\infty$). By requirement \eqref{eq:GroupActionHomogeneousSpace} in the definition of $\Gamma_{yx}$ we have $\Gamma^\alpha_{yx}F_x=F^\alpha_x$ and thus  $P(F,\Gamma^{\alpha})=0$ due to Lemma \ref{lem:PsiPol}. Hence $F^\alpha=F^{\sharp,\alpha}\in \Bs^{\gamma-\alpha}(\RR^d;\tT_\alpha)$ and from the structure condition \eqref{eq:StructureCondition} we obtain that $\partial^k F^\alpha_x =0$ for $k\in \NN^d$ with $0<|k|_\s<\gamma-\alpha$ (if any such $k$ exist). Thus
\begin{align*}
\|F_y^\alpha-\Gamma_{yx}^\alpha F_x\|_{\tT_\alpha}&=\|F^\alpha_y-F^\alpha_x\|_{\tT_\alpha}= \Big\|F_y^\alpha-F_x^\alpha-\sum_{k\in \NN^d_{<\gamma-\alpha}} \partial^k F_x^{\alpha}\,(y-x)^k\Big\|_{\tT_\alpha}\\
&=\Big\|F^{\sharp,\alpha}_y-F^{\sharp,\alpha}_x-\sum_{k\in \NN^d_{<\gamma-\alpha}} \partial^k F_x^{\sharp,\alpha}\,(y-x)^k\Big\|_{\tT_\alpha}\lesssim \|y-x\|_\s^{\gamma-\alpha}\,,
\end{align*}
where we applied Lemma~\ref{lem:Hoelder}. 

Let us now assume that we already know for some $\alpha \in A$ that for any $\alpha'\in A,\,\alpha'
>\alpha$
\begin{align}
\label{eq:BridgeTheoremInductionHypothesis}
\|F_y^{\alpha'}-\Gamma_{yx}^{\alpha'} F_x\|_{\tT_{\alpha'}}\lesssim   \|y-x\|_\s^{\gamma-\alpha'} \,.
\end{align}
We then show that \eqref{eq:BridgeTheoremInductionHypothesis} does also hold for all $\alpha'=\alpha$. To this end we reshape
\begin{align}
&F_y^\alpha-\Gamma_{yx}^\alpha F_x =F_y^\alpha -F^\alpha_x -\sum_{\alpha'\in A:\,\alpha'>\alpha} \Gamma^{\alpha}_{yx}F_x^{\alpha'} \nonumber
\\
&=F^{\sharp,\alpha}_y-F^{\sharp,\alpha}_x-\sum_{k:\,0<|k|_\s<\gamma-\alpha} \frac{1}{k!} \partial^k F^{\sharp,\alpha}_x (y-x)^k \label{eq:BridgeThmLocal1} \\
&+P(F,\Gamma^\alpha)_y-P(F,\Gamma^\alpha)_x+\sum_{k:\,0<|k|_\s<\gamma-\alpha} \frac{1}{k!} \partial^k \left(F^\alpha-P(F,\Gamma^\alpha)\right)_x (y-x)^k-\sum_{\alpha'\in A:\,\alpha'>\alpha} \Gamma^{\alpha}_{yx}F_x^{\alpha'}\,. \nonumber
\end{align}
Since \eqref{eq:BridgeThmLocal1} already decays in the right order due to Lemma \ref{lem:Hoelder} (and the assumption $\gamma\notin \AN$), we are only left with the last line which below we identify as the limit for $N\rightarrow \infty$ (in $\tT_\alpha$ for every $x,\,y$) of 
\begin{align*}
D^N &:=\sum_{j\leq N} D^N_j\\
 &:=\sum_{j\leq N} \Bigg[ \int \dd w\, \Big(\kk_{y-w}-\sum_{k\in \NN^d_{<\gamma-\alpha}}\frac{1}{k!}\partial^k \kk_{x-w} (y-x)^k\Big ) \Big( P(F,\Gamma^\alpha)_w -\sum_{\alpha'\in A:\,\alpha'
>\alpha} \Gamma^{\alpha}_{wx}F_x^{\alpha'} \Big)    
 \Bigg]\,.
\end{align*}
Indeed, we have the following three convergences:
\begin{align}
\label{eq:DN1}
&\sum_{j\leq N} \int \dd w\, \Big(\kk_{y-w}-\kk_{x-w}\Big ) P(F,\Gamma^\alpha)_w 
\\ &\hspace{60pt} \overset{N\rightarrow \infty}{\longrightarrow} P(F,\Gamma^\alpha)_y-P(F,\Gamma^\alpha)_x\,, \nonumber \\
& -  \sum_{j\leq N}\,\,  \sum_{0<|k|_\s<\gamma-\alpha} \int \dd w\, \frac{1}{k!} \partial^k \kk_{x-w} \Big(P(F,\Gamma^\alpha)_w-\sum_{\alpha'\in A:\,\alpha'>\alpha} \Gamma^\alpha_{wx} F^{\alpha'}_x\Big) (y-x)^k \label{eq:DN2}
\\ &\hspace{60pt} \overset{N\rightarrow \infty}{\longrightarrow} \sum_{0<|k|_\s<\gamma-\alpha} \frac{1}{k!} \partial^k (F^\alpha-P(F,\Gamma^\alpha))_x (y-x)^k \,, \nonumber\\
  &\sum_{j\leq N}  \int \dd w\, \Big(\kk_{y-w}-\kk_{x-w}\Big ) \sum_{\alpha'\in A:\,\alpha'>\alpha} \Gamma^{\alpha}_{wx}F_x^{\alpha'} \label{eq:DN3}
  \\ &\hspace{60pt} \overset{N\rightarrow \infty}{\longrightarrow} \sum_{\alpha'\in A:\,\alpha'>\alpha} \Gamma^{\alpha}_{yx}F_x^{\alpha'}-\sum_{\alpha'\in A:\,\alpha'>\alpha} \Gamma^{\alpha}_{xx}F_x^{\alpha'}=\sum_{\alpha'\in A:\,\alpha'>\alpha} \Gamma^{\alpha}_{yx}F_x^{\alpha'} \nonumber\,,
\end{align}
where we used (as in Remark \ref{rem:StructureConditionSmooth}) that $\Psi^{\leq N}=\sum_{j\leq N} \Psi^j$ is a Dirac sequence, for the second term in~\eqref{eq:DN2} we applied the structure condition \eqref{eq:StructureCondition}, and in~\eqref{eq:DN3} we used and the continuity of $\Gamma^\alpha_{\cdot x} F^{\alpha'}_x$ \footnote{A short computation shows that Definition \ref{def:Model} already implies (Hölder) continuity of the maps $y\mapsto \Gamma_{yx}^{\alpha}\tau$ for $\tau\in \tT$ and $x\in \RR^d$.}. Writing $D^N=\eqref{eq:DN1}+\eqref{eq:DN2}+\eqref{eq:DN3}$ we see the claimed convergence of $D^N$. We now show that, uniformly in $N$, 
\begin{align}
\label{eq:BoundDN}
\|D^N\|_{\tT_\alpha}\lesssim \|y-x\|_\s^{\gamma-\alpha}\,.
\end{align}
Note that we can reshape (with $\varDelta_{>0}:=\sum_{j>0} \varDelta_j$)
\begin{align*}
P(F,\Gamma^\alpha)_w &=\sum_{j>0}\dint \dd u \dd v\,\KK_{w-u} \kk_{w-v} \Gamma^\alpha_{vu}F_u  \\
&=\sum_{j>0} \dint \dd u \dd v\, \KK_{w-u} \kk_{w-v} \Gamma^\alpha_{vu} (F_u-\Gamma_{ux}F_x) +(\varDelta_{>0} \Gamma_{\cdot x}^\alpha F_x)_w\\
&=\sum_{j>0} \sum_{\alpha'\in A:\,\alpha'>\alpha} \dint \dd u \dd v\, \KK_{w-u} \kk_{w-v} \Gamma^\alpha_{vu} (F^{\alpha'
}_u-\Gamma_{ux}^{\alpha'}F_x) +\sum_{\alpha'\in A:\,\alpha'>\alpha}(\varDelta_{>0} \Gamma_{\cdot x}^\alpha F_x^{\alpha'})_w \,,
\end{align*}
where in the last line we used \eqref{eq:GroupActionHomogeneousSpace} for both terms  and  also that $\int \dd v\, \kk_v 1 =0$ for $j>0$ (by Lemma \ref{lem:PsiPol}) to cancel the $\alpha'=\alpha$ components. We can therefore rewrite $D_j^N$ as follows (with $R^{\gamma-\alpha}_{x;y-x}$ being the Taylor remainder as in \eqref{eq:TaylorRemainder})
\begin{align}
D^N_j&=\sum_{j\leq N}  \sum_{\alpha'\in A:\,\alpha'>\alpha}\int \dd w\, R^{\gamma-\alpha}_{x-w;y-x} \kk  \\
&\quad \times \left[ \sum_{i>0}\dint \dd u \dd v\, \Psi_{w-u}^{<i-1} \Psi_{w-v}^i \Gamma^\alpha_{vu}(F^{\alpha'}_u-\Gamma_{ux}^{\alpha'}F_x) +(\varDelta_{\leq 0} \Gamma_{\cdot x}^\alpha F_x^{\alpha'})_w\right] \nonumber  \\
&\oset{$(\ast)$}{=}\sum_{j\leq N}  \sum_{\alpha'\in A:\,\alpha'>\alpha}\int \dd w R^{\gamma-\alpha}_{x-w;y-x} \kk \left[ \sum_{i>0:\,i\sim j}\dint \dd u \dd v\, \Psi_{w-v}^{<i-1} \Psi_{w-u}^i \Gamma^\alpha_{vu}(F^{\alpha'}_u-\Gamma_{ux}^{\alpha'}F_x)\right] \label{eq:MainTheorem2}  \\
&\quad + \sum_{\alpha'\in A:\,\alpha'>\alpha} \int \dd w\, \Psi^{<N+1}_{w} \cdot R^{\gamma-\alpha}_{x-w;y-x} (\Delta_{\leq 0} \Gamma_{\cdot x}^\alpha F_x^{\alpha'}) \,,   \nonumber 
\end{align}
where in $(\ast)$ we used spectral support properties to restrict the inner sum to $i\sim j$ and the convolution-like structure to move the Taylor remainder to $\Delta_{\leq 0} \Gamma_{\cdot x}^\alpha F_x^{\alpha'}$ in the last term. The last term can be estimated by $\|x-y\|_\s^{\gamma-\alpha}$ via Lemma \ref{lem:Taylor} and Lemma \ref{lem:PsiPolScal} if we use that for every $k\in \mathbb{N}^d$ there is a $C>0$ such that
\begin{align}
\big\|\partial^k \left( \varDelta_{\leq 0} \Gamma^{\alpha}_{\cdot x} F^{\alpha'}_x \right)_{w-x+v^k_t(y-x)}\big\|_{\tT_\alpha}\lesssim  (1+\|w\|_{\s}^C)\,, \label{eq:BridgeThmLocal2}
\end{align}
which can be easily checked by direct computation. 
To handle the term \eqref{eq:MainTheorem2} we first estimate the sum in the square brackets: By Definition \ref{def:Model} and the induction hypothesis we have that
\begin{align*}
\sum_{\alpha'\in A:\,\alpha'>\alpha}\|\Gamma^\alpha_{vu}(F^{\alpha'}_u-\Gamma_{ux}^{\alpha'}F_x) \|_{\tT_\alpha}\lesssim  \sum_{\alpha'\in A:\,\alpha'>\alpha}\|v-u\|_\s^{\alpha'-\alpha}\cdot \|u-x\|_\s^{\gamma-\alpha'}\,.
\end{align*}
Lemma \ref{lem:PsiPolScal} thus yields
\begin{align}
\sum_{i:\,i\sim j}\sum_{\alpha'\in A:\,\alpha'>\alpha} &\left\| \dint \dd u \dd v\, \Psi_{w-v}^{<i-1} \Psi_{w-u}^i \Gamma^\alpha_{vu}(F^{\alpha'}_u-\Gamma_{ux}^{\alpha'}F_x) \right\|_{\tT_\alpha} \nonumber \\
&\lesssim \sum_{\alpha'\in A:\,\alpha'>\alpha}   2^{-j(\alpha'-\alpha)} 2^{-j(\gamma-\alpha')} \lesssim  2^{-j(\gamma-\alpha) }\,.
\label{eq:BridgeThmLocal3}
\end{align}
The rest of the estimate for \eqref{eq:MainTheorem2} then follows by the same line of reasoning as in Lemma \ref{lem:Hoelder}: Let $j'$ be such that $2^{-j'-1}< \|x-y\|_\s\leq 2^{-j'}$ and bound the sum in \eqref{eq:MainTheorem2} (up to a constant factor) using \eqref{eq:BridgeThmLocal3} by
\begin{align*}
\sum_{j\leq j'} \sum_{k\in \NN^d_{> \gamma-\alpha}} \|x-y\|_\s^{|k|_\s} \,2^{j(|k|_\s-(\gamma-\alpha))}+\sum_{N \geq j>j'} \sum_{k\in \NN^d_{<\gamma-\alpha}} \|x-y\|_\s^{|k|_\s} 2^{-j(\gamma-\alpha-|k|_\s)} \lesssim \|x-y\|_\s^{\gamma-\alpha}\,,
\end{align*}
where we applied Lemma \ref{lem:Taylor} in the low-frequency case and in both cases the $L^1$-bound from Lemma \ref{lem:PsiScal}. We have thus shown \eqref{eq:BoundDN} so that in total
\begin{align*}
\|F_y^\alpha-\Gamma^\alpha_{yx}F_x\|_{\tT_\alpha}=\lim_{N\rightarrow \infty} \|D_N\|_{\tT_\alpha}\lesssim  \|x-y\|_\s^{\gamma-\alpha}\,,
\end{align*}
which closes the induction and finishes the proof.
\end{proof}

\begin{remark}
 In \cite{ReconstructionBesov} the authors introduce a more general space $\DD^\gamma_{p,q}(\RR^d;\tT)$ which generalizes the space $\DD^\gamma(\RR^d;\tT)$ from Definition~\ref{def:Dgamma} (the special case $\DD^\gamma_{p,p}(\RR^d)$ was already introduced before by \cite{ProemelTeichmann}). Our approach would allow us to also define a ``Besov space'' $\mathscr{B}^\gamma_{p,q,\s}(\RR^d;\tT)$ (compare Remark \ref{rem:BesovGeneralIntegrability}), and we expect that it is equal to $\DD^\gamma_{p,q}(\RR^d)$. But for the sake of simplicity we restrict ourselves to the case $p=q=\infty$ in this work.  	
\end{remark}

{\footnotesize
\bibliography{DPSymmetry}
\bibliographystyle{alpha}}
\end{document}